\documentclass[11pt]{amsart}

\usepackage{graphicx}
\usepackage{amsfonts}
\usepackage{amsmath}
\usepackage{amssymb}
\usepackage{amscd, amsthm}
\usepackage{verbatim}
\usepackage{latexsym}
\usepackage{enumerate}
\usepackage{supertabular}
\theoremstyle{plain}

\newtheorem{Thm}[equation]{Theorem}
\newtheorem{Cor}[equation]{Corollary}
\newtheorem{Prop}[equation]{Proposition}
\newtheorem{lemma}[equation]{Lemma}
\newtheorem{definition}[equation]{Definition}
\numberwithin{equation}{section}

\newcommand{\cG}{\mathcal{G}}
\newcommand{\cO}{\mathcal{O}}

\newcommand{\GL}{{\mathrm{GL}}}
\newcommand{\PGL}{{\mathrm{PGL}}}
\newcommand{\SL}{{\mathrm{SL}}}
\newcommand{\Sp}{{\mathrm{Sp}}}
\newcommand{\SO}{{\mathrm{SO}}}
\newcommand{\Spin}{{\mathrm{Spin}}}
\newcommand{\PGSp}{{\mathrm{PGSp}}}
\newcommand{\PSp}{{\mathrm{PSp}}}

\newcommand{\B}{{\mathrm{B}}}
\newcommand{\C}{{\mathrm{C}}}
\newcommand{\E}{{\mathrm{E}}}
\newcommand{\rF}{{\mathrm{F}}}
\newcommand{\G}{{\mathrm{G}}}

\newcommand{\F}{\mathbb{F}}
\newcommand{\N}{\mathbb{N}}
\newcommand{\Z}{\mathbb{Z}}

\newcommand{\Aut}{\mathrm{Aut}\,}
\newcommand{\coker}{\mathrm{coker}\,}
\newcommand{\Inn}{\mathrm{Inn}\,}
\newcommand{\rk}{\mathrm{rk}\,}
\newcommand{\ad}{{\mathrm{ad}}}
\newcommand{\der}{{\mathrm{der}}}
\newcommand{\ind}{{\mathrm{ind}}}
\newcommand{\Ind}{{\mathrm{Ind}}}
\newcommand{\Hom}{{\mathrm{Hom}}}
\newcommand{\Gal}{{\mathrm{Gal}}}

\newcommand{\st}{{\mathrm{st}}}
\newcommand{\gen}{{\mathrm{gen}}}
\newcommand{\Fr}{{\mathrm{Fr}}}

\begin{document}

\title[Functoriality and the  Inverse Galois problem II]
{Functoriality and the Inverse Galois problem II: groups of type $B_n$ and $G_2$} 
\author[C.~Khare]{Chandrashekhar Khare}
\email{shekhar@math.utah.edu}
\address{Department of Mathematics \\
         UCLA \\
         Los Angeles CA 90095-1555 \\
         U.S.A.}\thanks{CK was partially supported by NSF grants DMS 0355528 and DMS  0653821,  the Miller Institute for Basic Research in Science, University of California Berkeley, and a Guggenheim fellowship.}
\author[M.~Larsen]{Michael Larsen}
\email{larsen@math.indiana.edu}
\address{Department of Mathematics\\
     Indiana University \\
     Bloomington, IN 47405\\
     U.S.A.}\thanks{ML was partially supported by NSF grant DMS  0354772.}
 \author[G.~Savin]{Gordan Savin}
\email{savin@math.utah.edu}
\address{Department of Mathematics \\
         University of Utah \\
         155 South 1400 East, Room 233 \\
           Salt Lake City, UT 84112-0090 \\
         U.S.A.}\thanks{GS was partially supported by NSF grant DMS 0551846.}

\maketitle 
\vskip 5pt

\section{Introduction}

\subsection{Earlier work}

Let $\ell$ be a prime. In our previous work \cite{KLS}, which generalised a result of Wiese \cite{W},  the Langlands functoriality 
principle was used to show that for every positive integer $t$ there exists a positive 
integer $k$ divisible by $t$ such that the finite  
simple group $\C_{n}(\ell^{k})=\PSp_{2n}(\mathbb F_{\ell^{k}})$ is the Galois 
group of an extension of $\mathbb Q$ unramified outside 
$\{\infty, \ell, q\}$ where $q\neq 2$ is a prime 
that depends on $t$. The construction is based on the following three steps. 

\begin{enumerate}

\item Starting 
with a cuspidal automorphic representation on the split group $\SO_{2n+1}$ constructed 
using the Poincar\'e series,  we use the global lift of
Cogdell, Kim, Piatetski-Shapiro and Shahidi \cite{CKPS} and results of 
Jiang and Soudry \cite{JS1} to obtain a self-dual 
cuspidal automorphic representation $\Pi$ of $\GL_{2n}({\mathbb A})$, with $\mathbb A$ the adeles of $\mathbb Q$,  such that the following 
three conditions hold: 
\begin{itemize}
\item $\Pi_{\infty}$ is cohomological. 
\item $\Pi_{q}$ is a supercuspidal representation of depth 0. 
\item $\Pi_{v}$ is unramified for all primes $v\neq \ell,q$. 
\end{itemize}

\item   The work of Kottwitz, Clozel, Harris-Taylor and Taylor-Yoshida   yields  the following theorem (see Theorem 3.6 of  \cite{Ty} or Theorem 1.1 of \cite{Harris}). We use the conventions and notations of \S 1 of \cite{Harris}.

 \begin{Thm}\label{reciprocity}

 Let $\Pi$ be a self-dual 
cuspidal automorphic representation $\Pi$ of $\GL_{m}({\mathbb A})$  such that $\Pi_{\infty}$ is cohomological.
Assume that for some finite place $v_0$ of $\mathbb Q$, $\Pi_{v_0}$ is square integrable. 
Then attached to $\pi$ and a choice of an embedding
 $\iota: \bar{\mathbb Q}  \hookrightarrow \bar{\mathbb Q}_\ell$, there is an  
irreducible $\ell$-adic representation $r'_{\Pi}:G_{\mathbb Q}\rightarrow 
\GL_{m}(\bar{\mathbb Q}_{\ell})$ of the Galois group $G_{\mathbb Q}$ of
 $\mathbb Q$ such that for all primes $v$  of $\mathbb Q$ of residue characteristic $\neq \ell$
we have: 
\[
WD_v(r'_{\Pi})^{\rm Frob-ss}= {\mathcal L}(\Pi_v \otimes | \ \ |_v^{1-m \over 2})
\]
Here $WD_v(r'_{\Pi})$  is the Weil-Deligne parameter of  $r'_{\Pi}|_{D_v}$ with $D_v$ a  decomposition group at $v$,  $\mathcal L$ is the normalised local Langlands correspondence, and $Frob-ss$ denotes Frobenius semisimplification. 

\end{Thm}

\noindent Remark: Let $\chi_\ell$ be the $\ell$-adic cyclotomic character. If $m=2n+1$ 
 we consider a twist  $r_\Pi= r'_\Pi \otimes \chi_\ell^{n}$ and note that we have
\[
WD_v(r_{\Pi})^{\rm Frob-ss}= {\mathcal L}(\Pi_v).
\] 
 

\item The last step consists of reducing $r'_{\Pi}$ modulo $\ell$. The parameter of $\Pi_{q}$ can be picked so that $r_{\Pi}(G_{\mathbb Q_{q}})$ is 
a metacyclic group deeply embedded in $r'_{\Pi}(G_{\mathbb Q})$ \cite{KW}.
That is, for some large positive integer $d$, $r'_{\Pi}(G_{\mathbb Q_{q}})$ is 
contained in every normal subgroup of $r'_{\Pi}(G_{\mathbb Q})$ of index less than or 
equal to $d$. This property is crucial to assure, using the main result of \cite{LP},  
that the reduction modulo $\ell$ is 
a simple group of type $\PSp_{2n}(\mathbb F_{\ell^{k}})$.  

\end{enumerate}

\smallskip

\subsection{Main theorem}

The purpose of this work is to extend these results and to construct finite simple groups 
of type $\B_{n}$ and $\G_{2}$ as Galois groups over $\mathbb Q$.
In the case of $\G_2$, our result depends on a recent technical improvement  
of Theorem \ref{reciprocity} due to Shin \cite{Sh} in the case that $m$ is odd. He shows that 
 we may drop the hypothesis of the existence of a place $v_0$ such that $\Pi_{v_0}$ is 
 square integrable. The resulting representation $r'_{\Pi}$ is semi-simple although it 
 is expected to be irreducible.

\smallskip

We can state our main theorem:

\begin{Thm} \label{main1}Let $t$ be a positive integer. We assume that $t$ is even 
if $\ell=3$ in the first case below: 
\begin{enumerate}
\item  Let $\ell$ be a prime. Then there exists an integer $k$ divisible by $t$ such that 
the simple group $\G_{2}(\mathbb F_{\ell^{k}})$ appears as a Galois group over
$\mathbb Q$. 
\item Let $\ell$ be an odd prime. Then there exists an integer $k$ divisible by $t$ such that 
 the finite simple group $\SO_{2n+1}(\mathbb F_{\ell^{k}})^{\der}$ or the 
finite classical group $\SO_{2n+1}(\mathbb F_{\ell^{k}})$
appears as a Galois group over $\mathbb Q$. 
\item If $\ell \equiv 3, 5\pmod{8}$,  then there exists an integer $k$ divisible by $t$ such that  the finite simple group 
$\SO_{2n+1}(\mathbb F_{\ell^{k}})^{\der}$ 
appears as a Galois group over $\mathbb Q$. 
\end{enumerate}
\end{Thm}

\subsection{Sketch of  proof} 

The construction of Galois groups in Theorem \ref{main1}  is based on the functorial lift from $\Sp_{2n}$ to $\GL_{2n+1}$ 
\cite{CKPS} plus the lift from $\G_{2}$ to $\Sp_{6}$ using the theta correspondence 
arising from the minimal representation of the exceptional group $\E_{7}$ 
(see \cite{Sa1} for a definition of the minimal representation). 
The main new technical difficulty in implementing the strategy of \cite{KLS}
in the present case,  is that $\GL_{2n+1}(\mathbb Q_{p})$ has self-dual 
supercuspidal representations only if $p=2$. Thus, while we can still construct a 
self-dual cuspidal  automorphic representation $\Pi$ of $\GL_{2n+1}$ which should give 
rise to our desired Galois groups, the local component $\Pi_{q}$ cannot be supercuspidal. 
 For groups of type $\B_{n}$ we can 
remedy the situation by requiring that the local component $\Pi_{2}$ be 
supercuspidal (which we pick to be of depth one).
 Existence of a global $\Pi$ with such local component $\Pi_{2}$ 
 is again obtained using the global lift from 
$\Sp_{2n}$ plus recently announced backward lift from $\GL_{2n+1}$ to $\Sp_{2n}$
by Jiang and Soudry \cite{JS2}. 
The local component $\Pi_{2}$ not only assures us of the existence of the $\ell$-adic 
representation $r_{\Pi}$, without using new results of Shin, but it also gives us a certain control over the Galois group 
obtained by reducing $r_{\Pi}$ modulo $\ell$. More precisely, 
 $\Pi_{2}$ can be picked so that the image 
of the local Langlands parameter is a finite group $I$ in $\GL_{2n+1}(\mathbb C)$
with the following properties: 
\begin{itemize}
\item $I/[I,I]\cong \mathbb Z/(2n+1)\mathbb Z$.
\item $[I,I]\cong (\mathbb Z/2\mathbb Z)^{2n}$. 
\end{itemize}
If $\ell\equiv 3,5\pmod{8}$ then the first property of $I$ implies that the 
Galois group is $\SO_{2n+1}(\mathbb F_{\ell^{k}})^{\der}$ and not 
$\SO_{2n+1}(\mathbb F_{\ell^{k}})$.
If $n=3$ then the second property of $I$ implies that $\Pi_{2}$ is not a lift
from $\G_{2}(\mathbb Q_{2})$ and the Galois group is not $\G_{2}(\ell^{k})$. 

\smallskip 
 
\textbf{ Correction to [KLS] :}  With the definition of the group of type $(n,p)$ in [KLS], to ensure
that $\bar{\rho}(D_q)$ in \S 5.2  be of type $(n,p)$ we  should ask that the $K$ of \S 3.3 
also contain ${\mathbb Q} (\zeta_\ell)$, in addition to the other conditions there. 
Alternatively, and better,  the definition
of a group of type $(n,p)$ in \cite{KLS} could be modified  (and made less restrictive) as in Definition \ref{def} below. All statements in [KLS] then go through with this altered definition, with obvious modifications in their proof.

\smallskip

 \textbf{Acknowledgments:}  We would like to thank Dick Gross and Guy Henniart
 for helping us with irreducible supercuspidal parameters and  Mark Reeder for 
 his help with small representations of reductive groups.

\section{Local discrete series parameters}

Let $k$ be a local field and $G$ a connected reductive and split group over $k$. Conjecturally, 
representations of $G(k)$ correspond to (certain) homomorphisms of the Weil-Deligne group 
\[
\phi : WD_{k} \rightarrow G^{\ast}(\mathbb C)
\]
into the Langlands dual group $G^{\ast}(\mathbb C)$. In this paper we shall be concerned 
with the following cases: 
 
 \[
 \begin{array}{|c||c|c|c|c|} \hline 
G &  \GL_{n} &\Sp_{2n} &  \PGSp_{6} & \G_{2} \\
\hline 
G^{\ast} & \GL_{n} & \SO_{2n+1} &\Spin_{7} & \G_{2}\\ 
\hline 
\end{array}
\]

If $G=\Sp_{2n}$, it will be convenient to realize the dual group
  as  $\SO(U)$ for some choice of a non-degenerate 
complex orthogonal space $U$ of dimension $2n+1$. Then a
 discrete series parameter for $\Sp_{2n}(k)$  is a homomorphism  
$\phi : WD_{k} \rightarrow \SO(U)$
such that under the action of $WD_{k}$ the orthogonal space  $U$ decomposes 
into irreducible summands 
\[
 U=U_{1} \oplus \cdots \oplus U_{s}
\]
where each $U_{i}$ is a non-degenerate orthogonal subspace of $U$.  
Moreover, if $\phi_{i}$ denotes the representation of $WD_{k}$ on $U_{i}$, then 
$\phi_{i}\cong \phi_{j}$ if and only if $i=j$. In other words, we are requiring that 
the image of $WD_{k}$ is not contained in a proper Levi factor in $G^{\ast}$. 

\smallskip 

Consider now the case $k=\mathbb R$. In this case the Weil-Deligne group is the same 
as the Weil group $W_{\mathbb R}$. 
 For every non-zero integer $a$ let  $\eta_{a}$ be a 
character of $W_{\mathbb C}\cong \mathbb C^{\times}$ defined by 
\[
\eta_{a}(z)=\left(\frac{z}{\bar z}\right)^{a}.
\]
Let 
\[
\phi(a)=\Ind_{W_{\mathbb C}}^{W_{\mathbb R}} \eta_{a}. 
\] 
This is an irreducible and orthogonal 2 dimensional representation of $W_{\mathbb R}$. 
Its determinant is the unique non-trivial quadratic character $\chi_{\infty}$ of 
$W_{\mathbb R}^{ab}\cong \mathbb R^{\times}$. If $a_{1}, \ldots, a_{n}$ are 
non-zero integers such that $a_{i}\neq \pm a_{j}$, then 
\[
\phi(a_{1}, \ldots , a_{n})\oplus \chi_{\infty}^{n}. 
\] 
is a discrete series parameter for the group $\Sp_{2n}(\mathbb R)$. 
Note that the choice of exponent - $n$ - in  the last summand is made so that the image of 
the parameter is contained in $\SO_{2n+1}(\mathbb C)$. Note that the parameter is 
determined by $a_{i}$ up to permutation of indices and change of signs. 
If $n=3$, then the image of 
the parameter is contained in $\G_{2}(\mathbb C)\subset \SO_{7}(\mathbb C)$ if and only if 
$$
a_{1}+a_{2}+ a_{3}=0
$$
for some choices of signs of $a_{i}$'s. Let $\sigma_{\infty}$ be a generic discrete 
series representation of $\Sp_{2n}(\mathbb R)$ (or of $\G_{2}(\mathbb R)$) corresponding
to this parameter. 

Let $\Pi_{\infty}$ be the lift of $\sigma_{\infty}$ to  
$\GL_{2n+1}(\mathbb R)$. The infinitesimal character of $\Pi_{\infty}$ is represented 
by a $2n+1$-tuple 
\[
(a_{1}, \ldots , a_{n},  -a_{1}, \ldots , -a_{n}, 0).
\]
In particular, $\Pi_{\infty}$ is cohomological, as defined by Clozel \cite{Cl}.

\section{Depth zero generic supercusipdal representations}

Let $q$ be an odd prime. 
Let $\Omega_{q'}$ denote the set of all complex roots of unity of order prime to 
$q$.  The Frobenius acts on $\Omega_{q'}$ by 
$$ 
F(\tau) = \tau^{q}
$$
for every $\tau$ in $\Omega_{q'}$. Note that all  $F$-orbits are finite. These 
orbits play a key role in the description of tame parameters. 

\begin{lemma} Let $\tau$ be  a root of 1 different from $\pm1$. Assume that the 
$F$-orbit 
of $\tau$ has $m$ different elements: 
$$
\tau , \tau^{q},  \ldots, \tau^{q^{m-1}}.
$$ 
If $\tau^{-1}$ is on this list, that is, if 
$\tau^{-1}=\tau^{q^{n}}$ for some $n< m$ then $m=2n$.  
\end{lemma}
\begin{proof} First of all, note that $0< n$ since $\tau\neq \pm 1$.  
Raising $\tau^{-1}=\tau^{q^{n}}$ to the $q^{n}$-th power gives 
$\tau=\tau^{q^{2n}}$. Since $\tau=\tau^{q^{k}}$ if and only if $k$ is a multiple of $m$,
and $0< 2n <2m$, it follows that $m=2n$, as claimed. 
\end{proof}

We are now ready to define irreducible tame self-dual 
parameters of $\Sp_{2n}(\mathbb Q_{q})$. 
Let $\mathbb Q_{q^{2n}}$ be the unique unramified extension of $\mathbb Q_{q}$ of degree $2n$. Then 
$$
\mathbb Q_{q^{2n}}^{\times}=\langle q\rangle \times \mathbb F_{q^{2n}}^{\times}
\times U_{1}
$$
where $U_{1}$ is the maximal pro $q$-subgroup of $\mathbb Q_{q^{2n}}^{\times}$. A character 
of $\mathbb Q_{q^{2n}}^{\times}$ is called tame if it is trivial on $U_{1}$. 
Let $\zeta_{2n}$ be a primitive root in $\mathbb F_{q^{2n}}^{\times}$.  
 Pick $\tau$, a complex  root of  $1$ such that 
 the $F$-orbit $\tau, \tau^{q}, \ldots $
 has precisely $2n$ distinct elements and $\tau^{q^{n}}=\tau^{-1}$. (For example, 
 $\tau$ can be picked a primitive root of order $q^{n}+1$.) Then $\tau$ defines 
a tame character $\eta$ of $\mathbb Q_{q^{2n}}^{\times}$ by 
$$ 
\begin{cases}
\eta(\zeta_{2n})=\tau \\
\eta(q)=1. 
\end{cases} 
$$
 Let $W_{\mathbb Q_{q}}$ and $W_{\mathbb Q_{q^{2n}}}$ 
 be the local Weil groups of $\mathbb Q_{q}$ and $\mathbb Q_{q^{2n}}$. Recall that 
$$ 
W_{\mathbb Q_{q}}/W_{\mathbb Q_{q^{2n}}}\cong \Gal(\mathbb F_{q^{2n}}/\mathbb F_{q}).
$$
Via the local class field theory we have an identification 
$W^{ab}_{\mathbb Q_{q^{2n}}}\cong \mathbb Q_{q^{2n}}^{\times}$.  
Note that $\eta\circ F^{i} \neq \eta$ for $1< i \leq 2n$
and $\eta\circ F^{n} = \bar \eta$. 
In particular, the character $\eta$ defines an irreducible, orthogonal  
$2n$-dimensional representation 
$$
\phi(\tau)= \Ind_{W_{\mathbb Q_{q^{2n}}}}^{W_{\mathbb Q_{q}}}(\eta). 
$$
of $W_{\mathbb Q_{q}}$.
 We note that the determinant of $\phi(\tau)$ is the unique unramified 
quadratic character $\chi_{q}$ of 
$W_{\mathbb Q_{q}}^{ab}\cong \mathbb Q_{q}^{\times}$. 

 Pick a  sequence $\tau_{1}, \ldots, \tau_{s}$
of roots in $\Omega_{q'}$ belonging to different $F$-orbits of order
 $2n_{1}, \ldots  , 2n_{s}$ such that $\tau^{q^{n_{i}}+1}=1$ for every $i$ and 
 $2n_{1}+ \cdots + 2n_{s}=2n$. Corresponding to this we have a 
  tame regular discrete series parameter for the split group $\Sp_{2n}(\mathbb Q_{q})$
 \[
\phi= \phi(\tau_{1}, \ldots , \tau_{s})\oplus \chi_{q}^{s}
 \]
where, as  in the case of real groups, the exponent $d$ is picked to assure that the image of the 
parameter is contained in $\SO_{2n+1}(\mathbb C)$.  Note that the image 
$\phi(I_q)$ of the inertia subgroup $I_q \subseteq W_{\mathbb Q_q}$ is contained in a maximal
torus of $\SO_{2n+1}(\mathbb C)$ and $\phi(F)$ is an elliptic element of the Weyl 
group. If $s=1$, for example, then the image of the inertia is a cyclic group generated 
by an element whose eigenvalues are 
\[
\tau, \tau^q, \ldots , \tau^{q^n}, \tau^{-1}, \ldots , \tau^{-q^n}, 1
\]
and $\phi(F)$ correspond to the Coxeter element in the Weyl group.     

  \begin{Prop} The image of a tame regular discrete series parameter $\phi=
  \phi(\tau_{1}, \ldots , \tau_{s})\oplus \chi_{q}^{s}$ of $\Sp_{6}(\mathbb Q_{q})$ is  
  contained in  $\G_{2}(\mathbb C)$ if and only if one of the two holds: 
  \begin{enumerate}
  \item $s=3$ and $\tau_{1} \tau_{2}\tau_{3}=1$, for some choices of 
  $\tau_{i}^{\pm}$ ($F$-orbit of $\tau_{i}$ consists of $\tau_{i}$ and $\tau_{i}^{-1}$). 
  \item $s=1$ and $\tau$ satisfies $\tau^{q^{2}-q+1}=1$. (Recall that
  $\tau$, a priori,  satisfies a weaker condition  $\tau^{q^{3}+1}=1$.)
  \end{enumerate}
  \end{Prop}
\begin{proof} 
The weights of the 7-dimensional representation of $\G_2(\mathbb C)$ are $0$ and six 
short roots. Pick three short roots $\alpha_1$, $\alpha_2$ and $\alpha_3$ such that 
$\alpha_1+\alpha_2+\alpha_3=0$.  If $t$ is a semi-simple element in $\G_2(\mathbb C)$, 
put $\lambda_i^{\pm}=\pm\alpha_i(t)$. Then $\lambda_1^{\pm}, \lambda_2^{\pm}, \lambda_3^{\pm}$ 
and $1$ are the eigenvalues of $t$ in the 7-dimensional representation. Note that 
$\lambda_1\lambda_2\lambda_3=1$. 
 
 If the parameter $\phi$ is contained in $\G_2(\mathbb C)$ then $\phi(F)$ corresponds to a Weyl 
 group element  in $\G_2$ of even order. Since 2 and 6 are only even orders of elements in 
  the Weyl group of $\G_2$, we see that $s=1$ or $3$.  If $s=3$ then $\phi(F)$ corresponds to $-1$ 
  in the Weyl group and the condition $\lambda_1\lambda_2\lambda_3=1$ translates into 
  $\tau_1\tau_2\tau_3=1$. If $s=1$ then $\phi(F)$ corresponds to the Coxeter element. We can 
  pick the Coxeter element (or alternatively the roots $\alpha_i$) so that it cyclically permutes the roots 
  \[
  \alpha_1, -\alpha_2, \alpha_3, -\alpha_1, \alpha_2,-\alpha_3. 
  \]
 On the other hand, $\phi(I_q)$ is generated by a semi-simple element $t$ with non-trvial 
 eigenvalues $ \tau, \tau^q, \tau^{q^2}, \tau^{-1}, \tau^{-q}, \tau^{-q^2}$
 which $\phi(F)$ permutes cyclically in the given order. In particular, the condition  $\lambda_1\lambda_2\lambda_3=1$ translates into $\tau^{1-q+q^2}=1$, as desired. 
 Conversely, if the parameter satisfies the conditions of (1) and (2) then we can factor $\phi$ 
 through $\G_2(\mathbb C)$ since -1 and the Coxeter element can be lifted from the Weyl 
 group to $\G_2(\mathbb C)$.    
\end{proof}

Let $\sigma_{q}$ be \emph{a} generic supercuspidal 
representation of $\Sp_{2n}(\mathbb Q_{q})$ (or of 
$\G_{2}(\mathbb Q_{q})$) corresponding, via DeBacker-Reeder,  to a tame parameter as above. 
Then the lift of $\sigma_{q}$ to $\GL_{2n+1}(\mathbb Q_{q})$ \cite{Sa2} is 
 $$ 
 \Pi_{1}\times \cdots \times \Pi_{s}\times \chi_{q}^{s}.
 $$
 This is a tempered representation parabolically induced from supercuspidal 
 representations $\Pi_{1}, \ldots , \Pi_{s}$ corresponding to irreducible tame 
 paramters $\phi(\tau_{1}), \ldots , \phi(\tau_{s})$ by the local Langlands 
 correspondence \cite{HT}. 
 We note that the recipe of DeBacker-Reeder \cite{DR}  involves picking a hyperspecial compact 
 subgroup of $\Sp_{2n}(\mathbb Q_{p})$. Since there are two non-conjugate hyperspecial 
 maximal compact subgroups here, there are two possible $\sigma_{q}$. They have 
the same lift to $\GL_{2n+1}(\mathbb Q_{p})$. 

\smallskip 
Of interest to us is the parameter of type $\phi(\tau)\oplus \chi_{q}$ where 
$q$ and $\tau$ are picked using the following lemma (Lemma 3.4 in \cite{KLS}): 

\begin{lemma} \label{deep} Given a positive integer $m=2n$, a prime $\ell$, 
a finite Galois extension $K$ of $\mathbb Q$, and positive integers $t$ and $d$, 
 there exists odd primes $p$ and $q$ such that 
\begin{enumerate}
\item Primes $\ell$, $p$ and $q$ are all distinct. 
\item The prime $p$ is greater than $d$. 
\item If $\SO_{2n+1}(\mathbb F_{\ell^{k}})$ contains an element of order $p$ then 
$\mathbb F_{\ell^{k}}$ contains $\mathbb F_{\ell^{t}}$. In particular, $t$ divides $k$.  
\item The prime $q$ splits completely in $K$. 
\item The order of $q$ in $\mathbb F_{p}^{\times}$ is exactly $m$. 
\end{enumerate}
\end{lemma}

We remark that if $n=3$, $\ell=3$, and $t$ is even, then no Ree group
$^2\G_2(\mathbb F_{3^{2f+1}})$ contains an element of order $p$.  Indeed,
\[^2\G_2(\mathbb F_{3^{2f+1}}) < \G_2(\mathbb F_{3^{2f+1}}) 
< \SO_7(\mathbb F_{3^{2f+1}}),\]
and $t$ does not divide $2f+1$.

Let $K$ be the composite of all Galois extensions 
of $\mathbb Q$ of degree $\leq d$ and ramified at $2,\ell$ and no other primes.
 This is a finite degree Galois extension of $\mathbb Q$ ramified at $2, \ell$ and no other primes. 
Let $p$ and $q$ be the primes given by Lemma \ref{deep} applied to this field $K$.  Let $\tau$ be a primitive $p$-th
root of $1$. Since the order of $q$ in $\mathbb F_{p}^{\times}$ is precisely $2n$, 
the $\Fr_{q}$-orbit of $\tau$ gives rise to a tame parameter
 $\phi(\tau)\oplus \chi_{q}$. Moreover, if $n=3$ then $\tau^{q^{2}-q+1}=1$ since $\tau$
 is of order $p$ and $p$, by construction, divides $\Phi_{6}(q)=q^{2}-q+1$. In particular, 
 this parameter is automatically a $\G_{2}$-parameter. In any case, we note that the 
 image of the inertia $I_{q}$ subgroup is a cyclic group of order $p$. The image of 
 the Weil group is a semi-direct product of the cyclic group $\mathbb Z/p\mathbb Z$ 
 and the cyclic group $\mathbb Z/2n\mathbb Z$. This group is also called 
 a \emph{metacyclic} group and denoted by $\Gamma_{2n,p}$. 
 
\section{Irreducible supercuspidal parameters}\label{depth1}
\label{self-dual}

As we have seen in the previous section, the image of a tame supercuspidal parameter 
$\varphi : W_{k} \rightarrow G^{\ast}$ is not irreducible when acting on the 
standard representation $U$ of $G^{\ast}$. In particular, the lift to $\GL_{n}(k)$ 
($n=\dim(U)$) of the corresponding supercuspidal representation is not  
supercuspidal. In order to remedy this, we need to introduce certain wildly ramified 
parameters. This will be done using (so-called) Jordan subgroups of the complex 
reductive group $G^{\ast}$. A Jordan subgroup $J$ of $G^{\ast}$ is an elementary 
abelian $p$-subgroup such that its normalizer $N$ in $G^{\ast}$ is a finite 
subgroup and $J$ is a minimal normal subgroup of $N$ see \cite{KT}, page 505. 
The following is a partial list of Jordan subgroups. 
\[
\begin{array}{|c|c|c|}
\hline 
 G^{\ast} & J & N/J\\
 \hline 
 \SO_{2n+1} & (\mathbb F_{2})^{2n} & S_{2n+1} \\
 \G_{2} & (\mathbb F_{2})^{3} & \SL_{3}(2) \\
 \hline 
 \end{array}
 \]
 Here $S_{2n+1}$ is the symmetric group of $2n+1$ letters. 
 We note that the conjugation action 
 of $N/J$ on $J$ given by the standard representation of $N/J$ on $J$. (In the first case we mean 
 by this the restriction of the permutation representation of $S_{2n+1}$ on 
 $\mathbb F_{2}^{2n+1}$ to the hyperplane given by $\sum_{i=1}^{2n+1}x_{i}=0$.)
 However  the extension of $N/J$ by $J$ is not necessarily split. 

\smallskip 

We shall now construct a map $\varphi: W_{\mathbb Q_{p}} \rightarrow G^{\ast}$
such that the image of the wild inertia is $J$ (in particular $p=2$) and 
the image of $W_{\mathbb Q_{p}}$ is 
an intermediate subgroup $J \subseteq I \subseteq N$ acting irreducibly on the standard representation 
of $G^{\ast}$. 

\smallskip 

Let us consider the case $G^{\ast}=\SO_{2n+1}(\mathbb C)$ first. 
Let us abbreviate $m=2n+1$, and let $\mathbb Q_{2^{m}}$ be the unramified extension of $\mathbb Q_{2}$
of degree $m$. Then 
\[
\mathbb Q_{2^{m}}^{\times}= \langle 2\rangle \times \mathbb F^{\times}_{2^{m}} 
\times U
\]
where $U$ is a pro-$2$ group with a filtration $U\supset U_{1} \supset U_{2} \ldots $
such that $U/U_{1}\cong \mathbb F_{2^{m}}$. Let $e$ be a primitive element 
in $\mathbb F_{2^{m}}$. Let $e_{i}=\Fr_{2}^{i-1}(e)$. 
Then $e=e_{1}, e_{2}, \ldots ,e_{m}$
give a basis of $\mathbb F_{2^{m}}$ over $\mathbb F_{2}$. 
In particular, we have fixed an isomorphism 
\[
U/U_{1} \cong (\mathbb F_{2})^{m}.
\]
In this way any character of $U/U_{1}$ can be viewed as an $m$-tuple of signs. Let 
$\chi$ be the character corresponding to the $m$-tuple $(-,-,+, \ldots , +)$.
We extend $\chi$ to $\mathbb Q_{2^{m}}^{\times}$ so that it is trivial on the first two factors. Since 
$W_{\mathbb Q_{2^{m}}}^{ab}\cong \mathbb Q_{2^{m}}^{\times}$ we can view $\chi$ as a character of $W(K)$. Define 
\[
\phi_{2}=\Ind_{W_{\mathbb Q_{2^{m}}}}^{W_{\mathbb Q_{2}}}(\chi).
\]
Since the conjugates $\chi\circ\Fr_{2}^{i}$, for $i=1,\ldots ,m$, are mutually distinct 
this representation is irreducible by Mackey's criterion.
 Since $\chi$ is quadratic the representation $\phi_{2}$ is also 
 delf-dual and, since $m$ is odd, it is orthogonal. Thus $\phi_{2}$ defines a self-dual 
 supercuspidal representation $\Pi_{2}$ of $\GL_{2n+1}(\mathbb Q_{2})$ by the local 
 Langlands correspondence. 
 
 \smallskip 
 
 For later purposes we need to describe the image of the representation $\phi_{2}$. Note that 
 the intersection of the kernels  of $\chi\circ \Fr_{2}^{i}$  is equal to $\Delta \mathbb F_{2}$, 
 the diagonal in $\mathbb F_{2}^{m}$.

 \begin{Prop} \label{two} 
 Recall that $m=2n+1$. Let $I$ be the image of 
 $W(\mathbb Q_{2})$ under the representation $\phi_{2}$. Then 
 \begin{enumerate}
 \item $I/[I,I]\cong \mathbb Z/m\mathbb Z$. 
 \item $[I,I]\cong \mathbb F_{2}^{m}/\Delta(\mathbb F_{2})$. 
 \item $I$ is contained in a special orthogonal group. 
 \item If $m=7$ then $I$ is not contained in $\G_{2}$. 
\end{enumerate}
\end{Prop}
\begin{proof} Since $W_{\mathbb Q_{2}}/W_{\mathbb Q_{2^{m}}}$ is a cyclic group of order $m$, 
in order to prove the first two statements, it suffices to show that the commutator 
is given by the image of $W_{\mathbb Q_{2^{m}}}$. Note that the commutator of $e_{i}$ and $\Fr_{2}$
in $W_{\mathbb Q_{2}}$ is equal to $e_{i}+e_{i+1}$, considered as an element of 
$U/U_{1}\cong \mathbb F_{2}^{m}$. Since $m$ is odd, these elements generate 
$\mathbb F_{2}^{m}/\Delta(\mathbb F_{2})$. The first two statements now follow. 
Since the determinant character is of order two and $I/[I,I]$
is odd, it has to be trivial on $I$. This shows the third statement. 
Finally, if $m=7$, then $\phi_{2}(e_{i})$ has eigenvalues $1$ (with multiplicity 5) and 
$-1$ (with multiplicity 2). These cannot be arranged as a family of type 
$\lambda_{1}^{\pm1}, \lambda_{2}^{\pm},\lambda_{3}^{\pm}$ and $1$ such that 
$\lambda_{1}\lambda_{2}\lambda_{3}=1$.  The proposition is proved. 
\end{proof}

\smallskip 
We now consider the Jordan subgroup in $G_{2}$. 
 the intermediate group $J\subseteq I\subseteq N$ in advance so that $I/J$ is the normalizer of an 
 elliptic torus in $N/J\cong \SL_{3}(2)$. In particular, if we identify $J$ 
 with $\mathbb F_{2^{3}}$ then $I/J$ can be identified as a semi-direct product of 
 $\Gal(\mathbb F_{2^{3}}/\mathbb F_{2})$ and $\mathbb F_{2^{3}}^{\times}$. 
  Since the order of $I/J$ is prime to $J$ one easily check that this extension splits. 
  Note that $I/J$ acts transitively on the set of non-trivial irreducible characters of $J$. 
  Let $\psi$ be a non-trivial additive character of $\mathbb F_{2}$. Then the composition of 
  the trace $Tr: \mathbb F_{2^{3}}\rightarrow \mathbb F_{2}$ is a character of 
  $\mathbb F_{2^{3}}$ such that its stabilizer in $I/J$ is 
  $\Gal(\mathbb F_{2^{3}}/\mathbb F_{2})$. 
 It follows, from Mackey's theory, that $I$ has 3 irreducible 
 faithful representations, of dimension 7, only one of which is self-dual. 
 
 \begin{Prop} \label{jordan}
 Let $\varphi :W_{\mathbb Q_{2}}\rightarrow \G_{2}$ 
  be a parameter with the image $I$. Let 
$\phi_{2}: W_{\mathbb Q_{2}}\rightarrow \GL_{7}(\mathbb C)$ 
obtained by natural inclusion $\G_{2}\subseteq \GL_{7}(\mathbb C)$. Then $\phi_{2}$ is  
 a self-dual,  irreducible representation of $W_{\mathbb Q_{2}}$. 
 \end{Prop}
 \begin{proof} This is easy. We know that any 
 irreducible representation of $I$ either has $J$ in the kernel or it is faithful, in which 
 case it is of the dimension $2^{3}-1=7$. There are three such representations, only one is
 is self-dual. 
 Since the restriction of the standard representation of $\G_{2}$ to $I$ 
  is faithful and self-dual, it must be isomorphic to the unique irreducible self-dual 
  representation of $I$ of the same dimension.  
\end{proof}

 \smallskip 
 It remains to show that the group $I$ can be obtained as the
  image of the Weil group $W_{\mathbb Q_{2}}$. 
 Let $L$ be the Galois extension of $\mathbb Q_{2}$ given as the totally ramified 
  extension of $\mathbb Q_{2^{3}}$ of degree $7$. In other words, 
 $L$ is the splitting field of the polynomial 
 \[
 X^{7}-2=0.
 \]
 Note that the Galois group of $L$ is isomorphic to $I/J$. Let $\varpi$ be a 
 uniformizer in $L$, $U\subseteq L^{\times}$ the maximal pro-$2$ subgroup and 
 $U\supseteq U_{1} \supseteq \ldots$ the usual filtration. Then 
  \[
 L^{\times} = \langle \varpi\rangle \times \mathbb F_{2^{3}}^{\times} \times U.
 \]
  Let $\chi$ be a character of $W_{L}^{ab}\cong L^{\times}$
 which is trivial on the the first two factors of $L^{\times}$ and a non-trivial 
 character of $U/U_{1}\cong \mathbb F_{2^{3}}$. Consider the induced 
 representation 
 \[
 \Ind_{W_{L}}^{W_{\mathbb Q_{2}}}(\chi).
 \]
This representation breaks up as a sum of three irreducible representations 
of dimension $7$, one of which is self-dual. 
$W_{K}$ be the kernel of this representation. Then the Galois group of 
$K$ over $\mathbb Q_{2}$ is isomorphic to $I$. In other words, we have constructed 
map $\varphi : W_{\mathbb Q_{2}}\rightarrow G^{\ast}$ with the image $I$.

\section{Local lift from $\G_{2}$ to $\PGSp_{6}$}

The dual group of $\PGSp_{6}(\mathbb Q_{p})$ is $\Spin_{7}(\mathbb C)$. The 
group $\Spin_{7}(\mathbb C)$ has a unique open orbit on the 8-dimensional spin  
representation. The stabilizer of a point in the open orbit is isomorphic to 
$\G_{2}(\mathbb C)$. This gives an embedding 
\[
f: \G_{2}(\mathbb C) \rightarrow \Spin_{7}(\mathbb C)
\]
of dual groups, indicating that there should be a functorial, but non-endoscopic, lift of 
representations from $\G_{2}(\mathbb Q_{p})$ to $\PGSp_{6}(\mathbb Q_{p})$, once 
local Langlands parametrizations for the two groups are established. The Langlands 
parameterization is essentially known for depth zero representations. We shall now 
spell out some special cases of our interest. In order to simplify notation let 
\[
\begin{cases} G=\G_{2}(\mathbb Q_{p}) \\
G'=\PGSp_{6}(\mathbb Q_{p}). 
\end{cases}
\]
Recall that a root of one $\tau$ such that $\tau^{p^{2}-p+1}=1$ defines a 
7-dimensional orthogonal parameter $\phi(\tau)\oplus \chi_{p}$ which is 
contained in $\G_{2}(\mathbb C)$. Therefore, it defines a generic 
supercuspidal representation \emph{denoted by $\sigma(\tau)$} of $G$ and, 
by composing this parameter with the inclusion $f$,  a generic 
supercuspidal representation $\sigma'(\tau)$ of $G'$. The 
representation $\sigma'(\tau)$, when restricted to $\Sp_{6}(\mathbb Q_{p})$, breaks 
up as a sum of two representations in the $L$-packet for 
the parameter $\phi(\tau)\oplus\chi_{p}$. 

We have the following: 
\begin{itemize}
\item The functorial lift of the supercuspidal 
representation $\sigma(\tau)$ is the supercuspidal representation $\sigma'(\tau)$. 
\item The functorial lift of the Steinberg representation $\st_{G}$ is the 
Steinberg representation $\st_{G'}$. 
\item Let $\sigma$ be an 
unramified representation $\sigma$ of $G$ corresponding to 
 a semi-simple conjugacy class (Satake parameter) $s\in \G_{2}(\mathbb C)$. Then 
 the lift of $\sigma$ is $\sigma'$, an unramified representation of $G'$
 corresponding to the parameter $s'=f(s)$. 
\end{itemize}

\smallskip 
 
 Although the local parametrizations for $G$ and $G'$ are not complete, a lift from 
 $G$ to $G'$ is given by a correspondence arising from the minimal 
  representation $\Sigma$ of the split, adjoint $\E_{7}(\mathbb Q_{p})$.  
 More precisely, if $\sigma$ is an irreducible representation of $G$, 
 then we define $\Theta(\sigma)$ to be the set of isomorphism classes of all 
 irreducible representations $\sigma'$ of $G'$ such that 
 $\sigma\otimes \sigma'$ is a quotient of $\Sigma$. 
 
 Let $\psi: U\rightarrow \mathbb C^{\times}$ be a Whittaker character for $G$, 
 where $U$ is a maximal unipotent subgroup of $G$. 
 Recall that a representation $\sigma$ of $G$ is  $\psi$-generic (or simply generic) if 
 $\sigma_{U,\psi}$, the space of $\psi$-twisted $U$-coinvariants, is nonzero. 
   We have the same definition for
  representations of $G'$ with respect to a Whittaker character 
  $\psi' : U' \rightarrow \mathbb C^{\times}$, where  
  $U'$ is a maximal unipotent subgroup of $G'$.
  Let $\Theta_{\gen}(\sigma)$ be 
 the subset of $\Theta(\sigma)$ consisting of generic representations. This set is 
 somewhat easier to determine. 
 
 \begin{Prop} Let $\sigma$ be an irreducible representation of $G$. Then 
 $\Theta_{\gen}(\sigma)\neq 0$ only if $\sigma$ is generic. Moreover, 
 $\Theta_{\gen}(\sigma)$ contains at most one element. 
 \end{Prop}
 \begin{proof} Let $\sigma'$ be in $\Theta_{gen}(\sigma)$. Then 
 $\sigma\otimes \sigma'$ is a quotient of $\Sigma$. Since $\sigma'_{U',\psi'}$
 is one dimensional, we see that $\sigma$ is a quotient of $\Sigma_{U',\psi'}$. 
 Since (\cite[Theorem 7.1]{Ga})
 \[
 \Sigma_{U',\psi'}=\ind_{U}^{G}(\psi), 
 \]
 as a $G$-module, it follows that $\sigma$ is indeed generic. Moreover, since 
 \[
 \Hom_{G}(\ind_{U}^{G}(\psi), \sigma)
 \]
 is one-dimensional for any generic representation $\sigma$ 
 (by uniqueness of the Whittaker functional), the second part 
 follows immediately. The proposition is proved. 
 \end{proof}
 
Given a generic representation $\sigma$,  the above proposition allows us to show
that $\Theta_{\gen}(\sigma)=\{\sigma'\}$ by simply showing that 
$\sigma\otimes \sigma'$ is a quotient of $\Sigma$.

 \begin{Prop} \label{local}
 Assume that $\sigma$ is an irreducible representation of $G=\G_{2}(\mathbb Q_{p})$, 
 belonging to either of the following two families: 
 \begin{enumerate}
 \item Supercuspidal representations $\sigma(\tau)$. 
 \item Generic unramified representations. 
 \end{enumerate}
 Then $\Theta_{\gen}(\sigma)\neq \emptyset$ and the unique representation in $\Theta_{\gen}(\sigma)$ is the 
 functorial lift of $\sigma$, as described above. 
 \end{Prop}
 \begin{proof}
 The supercuspidal representation $\sigma(\tau)$ is induced from a cuspidal 
representation $\rho$ of the finite group $\G_{2}(\mathbb F_{p})$, inflated to a 
hyperspecial maximal compact subgroup of $G$. Similarly, $\sigma'(\tau)$ is induced 
from a cuspidal representation $\rho'$ of $\PGSp_{6}(\mathbb F_{p})$. 
Gan shows in \cite{Ga} that $\rho\otimes \rho'$ is a summand of the minimal 
representation of the adjoint group $\E_{7}(\mathbb F_{p})$. By \cite{Sa1} the minimal 
representation of $\E_{7}(\mathbb F_{p})$ appears as the first non-trivial
$K$-type of the minimal representation of $\E_{7}(\mathbb Q_{p})$. (Here $K$ is a 
hyperspecial maximal compact subgroup in $\E_{7}(\mathbb Q_{p})$. In particular, 
$\E_{7}(\mathbb F_{p})$ is a quotient of $K$ by the first congruence subgroup, and 
the $K$-type is obtained by inflating the minimal representation of $\E_{7}(\mathbb F_{p})$
to $K$.)  It follows, by Frobenius reciprocity, that $\sigma(\tau)\otimes \sigma'(\tau)$ 
is a summand of $\Sigma$. This shows that 
$\Theta_{\gen}(\sigma(\tau))=\{\sigma'(\tau)\}$ as desired. 

  Finally, assume that $\sigma$ is an unramified representation. 
Let $B=TU$ be a Borel subgroup of $G$. Then $\sigma$ is a subquotient 
of $\Ind_{B}^{G}(\chi)$ for some unramified character $\chi$ of $T$. 
(The induction is normalized here.) Recall that any root $\alpha$ defines a 
co-root homomorphism $t \mapsto h_{\alpha}(t)$ from $\mathbb Q_{p}^{\times}$ 
into $T$. Let $\alpha_{1}, \alpha_{2}, \alpha_{3}$ be 
three short roots for $G$ such that 
\[
\alpha_{1}+\alpha_{2}+\alpha_{3}=0. 
\]
This choice is unique up to the action of the Weyl group of $\G_{2}$. We can now 
compose $\chi$ with the co-root homomophisms for $\alpha_{i}$.
 In this way we get 3 characters 
$\chi_{1}, \chi_{2},\chi_{3}$ of $\mathbb Q_{p}^{\times}$ such that 
$\chi_{1}\chi_{2}\chi_{3}=1$. 
Now, if $\sigma$ is generic then (and only then) the whole induced representation is 
irreducible. According to a result of Mui\'c (\cite{Mu} Proposition 3.1) this happens if and 
only if 
\[\chi_{i}\neq |\cdot|^{\pm 1} \text{ and } \chi_{i}/\chi_{j}\neq |\cdot|^{\pm 1}, 
1 \leq i < j \leq 3. 
\]
Next, consider the representation $\pi=\chi_{1}\times \chi_{2}\times \chi_{3}$ of
 $\GL_{3}(\mathbb Q_{p})$ (here we use the notation of Bernstein and 
 Zelevinski). Let $P=MN$ be a maximal parabolic of $G'$ 
 such that $M\cong \GL_{3}(\mathbb Q_{p})$ (see \cite{MaS}). Then 
 the local lift of $\sigma$ is the unique unramified quotient $\sigma'$ of the 
 representation of $G'$ obtained by inducing $\pi$.
 By a result of Tadi\'c (\cite{Ta} Theorem 7.1) this induced representation is irreducible 
 if and only if the same conditions as those of Mui\'c are satisfied. In other words, 
 an unramified representation $\sigma$ is generic if and only if its local lift 
 $\sigma'$ is, and both are equal to a fully induced principal series 
 representation. In particular, $\sigma'=\Ind_{P}^{G'}(\pi)$ (normalized induction). 
  By Frobenius reciprocity, we have 
 \[
 \Hom_{G\times G'}(\Sigma, \sigma\otimes \sigma')=
 \Hom_{G\times M}(\Sigma_{N}, \sigma\otimes \pi)
 \]
 where $\Sigma_{N}$ is the (normalized) Jacquet functor. Next, we recall that
 the minimal representation of $\E_{6}(\mathbb Q_{p})$ is a quotient 
 of $\Sigma_{N}$ \cite{MaS} and that $\sigma \otimes \pi$ is a quotient 
 of the minimal representation of $\E_{6}(\mathbb Q_{p})$ \cite{GaS2}.  
 It follows that $\Hom_{G\times M}(\Sigma_{N}, \sigma\otimes \pi)\neq 0$ and 
 $\sigma\otimes\sigma'$ is a quotient of $\Sigma$, as desired. 
 \end{proof}

\section{Global forms} \label{Global}

Recall that $G$ is the split $\Sp_{2n}$ or $\G_{2}$ over $\mathbb Q$.
 Fix a  prime $\ell$ and $q$ an odd prime different from $\ell$. 
By Theorem 4.5 \cite{KLS} there exists a globally 
generic cuspidal automorphic representation $\sigma$ of $G(\mathbb A)$ such that 
\begin{itemize}
\item $\sigma_{\infty}$ is a generic integrable discrete series representation. 
\item $\sigma_{q}$ is a tame supercuspidal generic representation;
$\sigma_q = \sigma(\tau)$ if $G = \G_{2}$. 
\item $\sigma_{v}$ is unramified for all $v\neq 2, q,\ell$. 
\item If $G=\G_2$ then $\sigma_{2}$ is unramified, and if $G=\Sp_{2n}$, 
$\sigma_{2}$ is the Jiang-Soudry descent of the self-dual 
supercuspidal representation of $\Pi_{2}$ introduced in Section \ref{self-dual}.  
\end{itemize}
The form $\sigma$ lifts to an irreducible, self-dual, automorphic representation 
$\GL_{2n+1}(\mathbb A)$ or  $\GL_{7}(\mathbb A)$, with \emph{trivial} 
central character. This uses the lift of Cogdell et al if $G$ is $\Sp_{2n}$. 
If $G$ is $\G_{2}$ we first use the exceptional theta lift \cite{GRS} to obtain a non-zero
 generic automorphic form $\sigma'$ on $\PGSp_{6}(\mathbb A)$. 
 The form $\sigma'$ is cuspidal if the lift of $\sigma$ to $\PGL_{3}(\mathbb A)$ 
(via the minimal representation of $E_{6}$)  is 0. This holds since $\sigma_{\infty}$ 
is a discrete series representation and it cannot appear as a local component in the 
lift from $\PGL_{3}(\mathbb A)$ \cite{GaS1}. Thus,  $\sigma'$ is a generic 
cuspidal automorphic representation and its local $p$-adic components are determined by
Proposition \ref{local}. The infinitesimal character 
of the real component $\sigma'_{\infty}$ is integral and regular by the matching 
of infinitesimal characters in \cite{HPS}. Next, we restrict $\sigma'$ to 
$\Sp_{6}(\mathbb A)$ and use the lift of Cogdell at al to obtain an 
automorphic representation $\Pi$ of $\GL_{7}(\mathbb A)$. 

\vskip 5pt 

Recall that $\chi_{q}$ is the unique non-trivial quadratic unramified character 
of the local Weil group. The local components of $\Pi$ satisfy: 
\begin{itemize}
\item $\Pi_{\infty}$ has a regular and integral infinitesimal  character. 
\item $\Pi_{q}$  has the parameter $\phi(\tau)\oplus \chi_{q}$. 
\item $\Pi_{v}$ is unramified for all $v\neq 2,q,\ell$. 
\item If $\ell\neq 2$ then $\Pi_{2}$ is unramified or, if $G=\Sp_{2n}$, 
 it has the irreducible parameter $\phi_{2}$. 
\end{itemize}
Note that if $\Pi_{v}$ is unramified than 
the eigenvalues of its Satake parameter are 
\[
\lambda_{1}^{\pm1}, \ldots , \lambda_{n}^{\pm1}, 1.
\]
 Moreover, if $G$ is $\G_{2}$ then we have one additional relation: 
 \[
 \lambda_{1}\lambda_{2}\lambda_{3}=1. 
 \]
 If $\Pi_{2}$ is the self-dual supercuspidal representation of $\GL_{2n+1}(\mathbb Q_{2})$
 with the parameter $\phi_{2}$ then the lift $\Pi$ is clearly supercuspidal. 
 If $\Pi_{2}$ is unramified then, since the parameter of $\sigma_{q}$ is not irreducible, 
 the global representation 
 $\Pi$ might not be cuspidal. We give a criterion which guarantees that it is. Note that the 
 conditions of the following proposition are automatically satisfied if $q$ and the 
 parameter $\phi(\tau)\oplus\chi_{q}$ are picked using Lemma \ref{deep}. 
 
 \begin{Prop} Assume that $\sigma$ is a globally generic cuspidal automorphic 
 representation of $\Sp_{2n}(\mathbb A)$, such that
  $\sigma_{v}$ is unramified at all (finite) primes $v\neq \ell, q$ and 
 $\sigma_{q}$ corresponds to the parameter $\phi(\tau)\oplus\chi_{q}$. 
 If $q$ splits in all quadratic extensions of $\mathbb Q$ ramified at $\ell$ and 
 no other primes then  $\Pi$, the global lift of $\sigma$ to $\GL_{2n+1}(\mathbb A)$,  is cuspidal. 
\end{Prop}
\begin{proof}
If $\Pi$ is not cuspidal then by \cite{CKPS} and forced by 
the local parameter of $\Pi_{q}$ we have an isobaric sum 
$$ 
\Pi = \Sigma\boxplus \chi
$$
where $\Sigma$ is a cuspidal self-dual automorphic representation of 
$\GL_{2n}(\mathbb A)$ and $\chi$ is a quadratic character 
of $\GL_{1}(\mathbb A)$.  The local component $\chi_{v}$ of $\chi$ is 
clearly unramified for all primes $v\neq \ell,q$. It is also unramified 
and non-trivial at $q$ since it corresponds, via the local class field theory, 
 to the one-dimensional summand of the 
parameter of $\Pi_{q}$ (denoted by the same symbol $\chi_{q}$). 
By the global 
class field theory $\chi$ corresponds to a quadratic extension of $\mathbb Q$ 
ramified at $\ell$ and no other primes, and such that $q$ is inert. This is a contradiction.  
\end{proof}

\section{Reductive Groups}

Let $G$ be a connected reductive group over an algebraically closed field of characteristic 
different from 2, and let $T$ be a maximal 
torus in $G$. A representation $V$ of $G$ is called almost miniscule if $V^{T}\neq 0$
and the Weyl group
acts transitively on the set of non-trivial weights of $V$. These representations can be 
easily classified for an almost simple $G$. Assume first that the field characteristic is 0. 
Since $V^{T}\neq 0$ then, by Lie algebra action,
 $V$ contains a root as a weight. All non-zero weights are now Weyl-group conjugates of that root. 
 Therefore, if $G$ is simply laced then $V$ is the adjoint representation of $G$. If $G$ is multiply laced then
the weights are all short roots. We tabulate possible cases: 
\[
\begin{array}{|c|c|c|}
\hline 
 G & \dim(V) & \dim(V^{T})\\
 \hline 
 \B_{n} & 2n+1 & 1 \\
\C_{n} & 2n^{2}-n-1 & n-1\\
 \G_{2} & 7 & 1 \\
 {\rm F}_{4} & 26 & 2 \\
 \hline 
 \end{array}
 \]
 It is interesting to note that dimension of $V^{T}$ is equal to the number of short 
 simple roots. The Weyl group acts, naturally, on $V^{T}$. The action of 
  long root reflections is trivial,  while $V^{T}$ is a reflection representation for the 
 subgroup generated by simple short root reflections.

\begin{Prop} Let $G$ be an almost simple group over an algebraically closed field 
of characteristic $\ell>2$. Then: 
\begin{enumerate} 
\item Any miniscule representation is isomorphic to a Frobenius twist of a representation 
with a miniscule weight as the highest weight. 
\item Assume first that  $\ell \neq 3$ if $G=\G_2$. Then any almost miniscule representation is a  
Frobenius twist of the almost miniscule representation 
with the highest weight equal to a short root. If $\ell =3$ and $G=\G_2$ then there is one additional 
family of miniscule representations. It consists of Frobenius twists of the representation with 
the highest weight equal to a long root. 
\end{enumerate}
\end{Prop}
\begin{proof} Let $\alpha_1, \ldots , \alpha_r$ be the simple roots for $G$. We shall characterize 
miniscule (and then almost miniscule) representations with the 
 highest weight $\lambda$ such that $0\leq \langle \lambda, \alpha_i^{\vee} \rangle 
\leq \ell -1$ for all $i$.  
The general case can be easily deduced from  the Steinberg tensor product theorem. 
 
For any root $\alpha$, the group $G$ contains a subgroup 
 isomorphic to (a quotient of) $\SL_2$. If $\langle \lambda,\alpha^{\vee}\rangle = n \leq \ell -1$ 
 then the action of $\SL_2$ on the highest 
weight vector will give rise to the weights $\lambda, \lambda - \alpha, \ldots \lambda-n\alpha$. 
By examining the lengths of these weights we see that only the last one is in the Weyl group 
orbit of $\lambda$. This forces $n=0 $ or $1$ if the representation is to be miniscule. 
In particular, if we write $n_i= \langle \lambda, \alpha_i^{\vee}\rangle$, then $n_i = 0 $ or $1$ for all $i$. 
Next, we claim that only one $n_i$ could be 1. Otherwise, we can pick a path in the 
Dynkin diagram $\alpha_i, \alpha_{i+1}, \ldots, \alpha_j$ such that $n_i = n_j =1$ and 
$n_k=0$ for any $\alpha_k$ between $\alpha_i$ and $\alpha_j$. 
Consider the root  $\alpha= \alpha_i + 
\alpha_{i+1} + \cdots + \alpha_j$. Since $\langle \lambda, \alpha^{\vee}\rangle =2$,  this is a contradiction. (Note that we have just used the condition $\ell\neq 2$.)
Finally, if the weight $\lambda$ is fundamental but not miniscule then there exists a positive 
root $\alpha$ such that $\langle\lambda,\alpha^{\vee}\rangle =2 $. Again, a contradiction. 

The proof of (2) is similar, except when $\langle\lambda,\alpha^{\vee}\rangle = 2$. Then 
$\lambda, \lambda-\alpha$ and $\lambda-2\alpha$ are weights with $\lambda-\alpha$ the 
shortest length among the three weights. Since $0$ is the only other orbit of weights, 
we must have $\lambda=\alpha$.  If $\alpha$ is a long highest root, and the root system is of 
the type $B_n$, $C_n$ or $F_4$ then there exists a short root $\beta$ such that 
$\langle\alpha, \beta^{\vee}\rangle =2$. This implies that the short root $\alpha-\beta$ is also a weight, 
and the representation cannot be almost miniscule. 
If $\ell =3$ and $G=\G_2$ then this argument breaks down. The adjoint representation breaks up as
$14=7+ 7'$ where $7$ is the representation whose non-trivial weights are short roots, while 
$7'$ is a representation whose non-trivial weights are long roots. 
 \end{proof}

We also note that the 26-dimensional almost miniscule representation of ${\rm F}_4$ is not 
irreducible. It breaks up as $26=25+1$. The main result of this section is 
 the following:

\begin{Prop} \label{almost}
Let $G$ be a connected reductive group over an algebraically closed field of characteristic 
different from 2. Let $V$ be an irreducible and faithful representation of $G$ of 
dimension $2n+1$, preserving a non-degenerate bilinear form. Assume that there exist 
$2n$ different weights in $V$ permuted cyclically by a Weyl group element. 
Then $G=\SO_{2n+1}$ or $\G_{2}$ if $n=3$. 
\end{Prop}
\begin{proof}  
 Let $Z$ be the connected component of the center of $G$. The characters 
of $Z$ form a lattice. In particular, the only self-dual character is the trivial character. 
This shows, since $V$ is faithful, that $Z$ is trivial. Therefore $G$ is semi-simple. 
Since weights of a self-dual representation come in pairs $\{\mu, -\mu\}$, and the 
Weyl group preserves the length of weights, the weight outside the cycle must be 
trivial, so $\dim(V^{T})=1$.
  Next, let $G_{1}\times \cdots \times G_{k}$ be a product of almost simple 
groups isogenous to $G$ such that $V=V_{1}\otimes \cdots \otimes V_{k}$ 
is a tensor product of irreducible,  \emph{non-trivial} 
representations of $G_{1}, \ldots, G_{k}$. 
Since $V^{T}\neq 0$, the zero weight must appear in each $V_{1}, \ldots , V_{k}$. 
But then $V$ can be almost miniscule only if $k=1$. From the list of almost 
miniscule representations of almost simple groups we see that $V$
 must be the standard representation of $\SO_{2n+1}$ or $G_{2}$ or a Frobenius twist of it, 
 if $\ell \neq 3$. If $\ell =3$ then ${\rm F}_4$ has an almost miniscule representation of dimension $25$. 
 However, since ${\rm F}_4$ has no element of order 24 in its Weyl group, there is no element 
 permuting all non-trivial weights. This completes the proof. 
\end{proof}

\section{Galois representation}\label{Galois}

Let $\Pi$ be the self-dual cuspidal automorphic representation of $\GL_{2n+1}(\mathbb A)$ 
constructed by lifting from $\Sp_{2n}$, or from $\G_{2}(\mathbb A)$ 
in Section \ref{Global}.  In particular, 
\begin{itemize}
\item The infinitesimal character of $\Pi_{\infty}$ is regular and integral. 
\item The local component $\Pi_{v}$ is unramified for all $v\neq \ell, q$
\item The local parameter of $\Pi_{q}$ is $\phi(\tau)\oplus\chi_{q}$ where 
$q$ splits in any quadratic extension of $\mathbb Q$ ramified at $\ell$ only. 
\item If $\ell\neq 2$ then $\Pi_{2}$ is unramified or, if $G=\Sp_{2n}$, 
the local parameter is the irreducible representation $\phi_{2}$. 
\end{itemize}

By the generalization of Theorem \ref{reciprocity} due to Shin, 
and the remark after it, one can attach a 
semi-simple Galois representation to $\Pi$:
\[
r_{\Pi}: G_{\mathbb Q} \rightarrow \GL_{2n+1}(\bar{\mathbb Q}_{\ell})
\]
such that for every 
prime $v\neq \ell$ the restriction of $r$ to the decomposition group $D_{v}$ gives the Langlands parameter of 
$\Pi_{v}$, up to a Frobenius semi-simplification.

[Note  that  when the Weil-Deligne parameter ${\mathcal L}(\Pi_v)$ for all $v\neq \ell$
has monodromy $N=0$, we can state Theorem \ref{reciprocity}  as : 
\[
 r_{\Pi}|_{D_v}^{\rm Frob-ss}= {\mathcal L}(\Pi_v) 
 \]
  and where 
${\mathcal L}(\Pi_v)$  may be regarded, using the embedding $\iota:\bar {\mathbb Q}  \hookrightarrow \bar {\mathbb Q}_\ell$, as a representation of the decomposition group $D_v$ at $v$ with values in $\GL_{2n+1}(\bar{\mathbb Q}_\ell)$.]

\smallskip 

Let us concentrate on 
the prime $q$. Here $\Pi_{q}$ is the lift of a supercuspidal representation of
 $\Sp_{2n}(\mathbb Q_{q})$ whose parameter, when restricted to the inertia group 
 $I_{\mathbb Q_{q}}$, is a direct sum of $2n+1$ one dimensional characters. One trivial and 
 $2n$ other cyclically permuted by $\Fr_{q}$. It follows that $r_{\Pi}(\Fr_{q}^{2n})$ commutes 
 with $r_{\Pi}(I_{q})$. This shows that $r_{\Pi}(\Fr_{q}^{2n})$ and $r_{\Pi}(\Fr_{q})$ must
be  semi-simple. In other words, $r$ gives exactly the parameter of $\Pi_{q}$. The same 
argument shows that the restriction of $r_{\Pi}$ to $D_{2}$ gives the parameter of $\Pi_{2}$
if $\Pi_{2}$ is supercuspidal.

\begin{Prop} If $\Pi$ satisfies the above conditions at places $\infty$, $2$ and $q$  
then the Galois representation $r_{\Pi}$ is irreducible and orthogonal. 
\end{Prop}
\begin{proof} If $\Pi_{2}$ is supercuspidal, then the local parameter is 
irreducible and $r_{\Pi}$ is irreducible. 
Thus, assume that $\Pi_{2}$ is unramified (or $\ell=2$). If
$r_{\Pi}$ is reducible then $r_{\Pi}$ has two irreducible summands 
of dimensions $2n$ and 1. Since the   
 eigenvalues of $r_{\Pi}(\Fr_{v})$ are $1$, $\lambda_{i}^{\pm}$, ($1\leq i\leq n$) the 
representation $r_{\Pi}$ is self-dual. In particular, the one-dimensional summand is 
a quadratic character $\chi$ ramified at $\ell$ only and such that $\chi(\Fr_{q})=-1$. 
This implies that there exists a quadratic extension ramified at $\ell$ only and such that
 $q$ stays inert. This is a contradiction since 
$q$ is picked so that it splits completely in every quadratic extension ramified at 
$\ell$ only. Therefore $r_{\Pi}$ is irreducible. Since it is self-dual and of odd dimension 
it is orthogonal as well. 
\end{proof}

\section{Zariski closure}

We shall now assume that the odd prime $q$ and the parameter 
$\phi(\tau)\oplus \chi_{q}$ of the local component $\Pi_{q}$ of $\Pi$ are picked using 
Lemma \ref{deep}. In particular, the representation $r_{\Pi}$ is unramified at 
all primes different from $2$, $\ell$ and $q$, and $r_{\Pi}(I_{q})$, the image of the 
inertia subgroup is of order $p$.   
Let $\Gamma=r_{\Pi}(G_{\mathbb Q})$.  
If $d>1$ is an integer, let $\Gamma^{d}$ be the intersection of all normal subgroups 
of $\Gamma$ of index $\leq d$. The following is crucial:

\begin{lemma} \label{deep2}
Assume that  the local component $\Pi_{q}$ is constructed by means of 
Lemma \ref{deep}. Then $r_{\Pi}(D_{q})$ is contained in $\Gamma^{d}$. 
\end{lemma}
\begin{proof}
Let $\Gamma'$ be a normal subgroup of index $\leq d$. We must show that the image of 
$D_{q}$ lands in $\Gamma'$. Let $L$ be the Galois extension of $\mathbb Q$ corresponding to
$\Gamma'$. Obviously, $L$ is unramified at all primes except $2, \ell$ and $q$. Moreover, 
since $r_{\Pi}(I_{q})$ is of order $p$ and $p>d$, it follows that $L$ is unramified at $q$ as well. 
Thus, $L$ is contained in $K$, the compositum of all Galois extensions ramified at $2,\ell$ and 
of degree $\leq d$. This implies that $q$ splits completely in $L$ (by Lemma \ref{deep}) 
and $r_{\Pi}(D_{q})$ is therefore contained in $\Gamma'$, as desired. 
\end{proof}

The above lemma shows that if $\Pi_{q}$ is constructed by means of Lemma \ref{deep}
then $\Gamma_{2n,p}$, the image of the decomposition group $D_{q}$,
 sits \emph{deeply embedded} in 
$\Gamma=r_{\Pi}(G_{\mathbb Q})$. This property is crucial in controling the 
size of the image $\Gamma$ of the Galois group. 

 \begin{Thm}\label{zariski}
  There exists a function $I : \mathbb N\rightarrow \mathbb N$ 
 such that if $\Pi$ is a self-dual cuspidal automorphic representation of 
 $\GL_{2n+1}(\mathbb A)$, as in Section \ref{Galois}, and such that the local parameter of $\Pi_{q}$ 
 is constructed by means of Lemma \ref{deep} with $d> I(n)$ then 
 the Zariski closure of $r_{\Pi}(G_{\mathbb Q})$ is isomorphic to 
 $\SO_{2n+1}(\bar{\mathbb Q}_{\ell})$ if 
$n\neq 3$ and $\SO_{7}(\bar{\mathbb Q}_{\ell})$
or $\G_{2}(\bar{\mathbb Q}_{\ell})$ if $n=3$.
\end{Thm} 
\begin{proof}
 Let $G$ be the Zariski closure of $r_{\Pi}(G_{\mathbb Q})$.  Since $r_{\Pi}$ is irreducible 
this is a reductive group. Let $G^{\circ}$ be 
its connected component. Recall that the image of the inertia subgroup $I_{q}$ 
contains an element $s$ in $\GL_{2n+1}(\bar{\mathbb Q}_{\ell})$ of order $p$. 
We want to show that $s$ is contained in $G^{\circ}$. We need the 
following (Lemma 6.3 in \cite{KLS}). 

\begin{lemma}
There exists a function $J\colon \N\to\N$ such that every integer $n>0$ and every algebraic 
subgroup $G\subset \GL_n$ over a field of characteristic zero, there is normal subgroup 
$G_{1}\subseteq G$ of index $\le J(n)$ containing $G^{\circ}$ such that 
$G_{1}/G^{\circ}$ is abelian. 
\end{lemma}
If we pick $d>J(n)$ then, by Lemma \ref{deep2},
 the image of $D_{q}$ must be contained in $G_{1}$. Since $\Gamma_{2n,p}$
  is a semi-direct product of $\mathbb Z/p\mathbb Z$ and 
$\mathbb Z/2n\mathbb Z \subset \Aut(\mathbb Z/ p\mathbb Z)\cong 
\mathbb Z/(p-1)\mathbb Z$ one easily sees that the commutator subgroup of 
$r_{\Pi}(D_{q})$ is $r_{\Pi}(I_{q})\cong \mathbb Z/p\mathbb Z$. This shows that 
the projection of $r_{\Pi}(D_{q})$ to the abelian quotient 
$G_{1}/G^{\circ}$ must contain $s$ in the kernel. In other words, $s$ is in $G^{\circ}$. 

\smallskip 

Recall that the eigenvalues of $s$ are 
\[
\tau,\tau^{q},\ldots ,\tau^{q^{n-1}}, \tau^{-1},\tau^{-q},\ldots,\tau^{-q^{n-1}},1.
\]
where $\tau$ is a $p$th root of one.  Moreover, these eigenvaules are 
\emph{distinct}. It follows that the 
 the centralizer of $s$ in $\GL_{2n+1}$ is a torus. Thus, $s$ is 
a regular semi-simple element in $G^{\circ}$ and the centralizer of 
 $s$ in $G^{\circ}$ is a maximal torus $T$. Next, 
 note that the centralizer of $s$ in $G^{\circ}$ is the same as the centralizer of 
 $r_{\Pi}(\Fr_{q})\cdot s\cdot r_{\Pi}(\Fr_{q}^{-1})$
 (since it is so in $\GL_{2n+1}$).  It follows that $r_{\Pi}(\Fr_{q})$ normalizes $T$.  
 We need the following lemma. 
 
 \begin{lemma} There exists a function $J'\colon \N\to\N$ such that every integer $n>0$ 
 and every reductive group $G$ of rank $n$ over an algebraically closed field of characteristic zero, there is normal subgroup 
$G_{1}\subseteq G$ of index $\le J'(n)$ containing $G^{\circ}$ such that conjugation of 
$G^{\circ}$ by any element in $G_{1}$ is inner. 
\end{lemma}
\begin{proof}
 The conjugation action of $G$ on $G^{\circ}$ gives a homomorphism 
 from a finite group $G/G^{\circ}$ to $\Aut(G^{\circ})/\Inn(G^{\circ})$. 
 The latter group is contained in the group of automorphisms of a quadruple 
 $(X, \Delta, X^{\vee}, \Delta^{\vee})$ where $X$ and $X^{\vee}$ are the 
 character and co-character lattices of a maximal torus 
 $T$ in $G^{\circ}$, and $\Delta$ a choice of simple roots for $G^{\circ}$.
 It follows that $G/G^{\circ}$ maps into $\GL(X)\cong \GL_{n}(\mathbb Z)$. Since 
 there is a torsion-free congruence subgroup in $\GL_{n}(\mathbb Z)$, the order of 
 any finite subgroup in $\GL_{n}(\mathbb Z)$ is bounded by the index of the 
 congruence subgroup. This proves the lemma. 
 \end{proof}
 
  Now, if we pick $d>J'(n)$ then $r_{\Pi}(D_{q})$ 
 must be contained in $G_{1}$. This shows that the normalizer of the torus $T$ in 
 $G^{\circ}$ contains an element of order $2n$ which cyclically permutes
the weights corresponding to the $2n$ eigenvalues of $s$ different from $1$. 
 This shows that $r_{\Pi}$, under the action of $G^{\circ}$, has at most two irreducible summands, 
 of dimension $2n$ and $1$. In fact, since $G^{\circ}$ is a normal subgroup and $G$ 
 acts irreducibly, a simple argument shows that $G^{\circ}$ must act irreducibly as well. 
 It follows that $G^{\circ}$ satisfies conditions of Proposition \ref{almost}, so it must 
 be either $\SO_{2n+1}$ or $\G_{2}$.
  It remains to show that $G=G^{\circ}$. This is easy. First, $r_{\Pi}$ is an 
 irreducible $2n+1$-dimensional representation of $G$ which restricts to the standard 
 representation of $G^{\circ}$.  Second, since $\SO_{2n+1}$ and 
 $\G_{2}$ have no outer automorphisms, every connected component of $G$
  contains an element commuting with $G^{\circ}$. Combining the two, it follows that 
  any connected component of $G$ contains a homothety by a scalar $\lambda$. 
  On the other hand, since $r_{\Pi}(\Fr_{v})$ has 1 as one of the eigenvalues for 
  almost all primes $v$, by \v Cebotarev's density the function $\det(1-r_{\Pi}(g))$ must be 
 0 for all elements $g$ in $G$.  In particular, $\lambda$
  must be equal to $1$ and  this shows that $G=G^{\circ}$. The proposition is proved
  with $I(n)=\max(J(n),J'(n))$. 
  \end{proof}
  
  According to Theorem \ref{zariski} there are two possibilities for the Zariski closure if $n=3$. 
   The following two corollaries give us  more precise statements in this case. 
   
  \begin{Cor} Assume that $\Pi$ comes from $\G_{2}(\mathbb A)$. Then 
  the Zariski closure of $r_{\Pi}(G_{\mathbb Q})$ is $\G_{2}(\bar{\mathbb Q}_{\ell})$. 
  \end{Cor}
\begin{proof} Let $G$ be the Zariski closure. We know that $G$ is either 
$\SO_{7}$ or $\G_{2}$. It suffices to show that the rank of the maximal torus $T$ is 
at most $2$.  
To see this, recall that the eigenvalues of $r_{\Pi}(\Fr_{v})$ are 
1, $\lambda_{i}^{\pm}$  ($i=1,2,3$) for $v\neq q, \ell$. Therefore the characteristic polynomial of 
$r_{\Pi}(\Fr_{v})$ is $f(x)=(x-1)g(x)$ with 
\[
g(x)=x^{6}+ ax^{5}+bx^{4}+cx^{3}+bx^{2}+ax+1
\]
 a palindromic polynomial of degree $6$. Moreover, 
 the condition $\lambda_{1}\lambda_{2}\lambda_{3}=1$ gives one algebraic relation on 
 the three coefficients $a$, $b$ and $c$. By \v Cebotarev's density theorem, the same 
 holds for the characteristic polynomial of all elements in the image of $r$ and, 
 therefore, for all elements in $T$. In particular, the dimension of $T$ is less than or equal to 2. 
 \end{proof}
 
 \begin{Cor} Assume that $\Pi_{2}$, the local component of $\Pi$, is the irreducible 
 self-dual cuspidal representation of $\GL_{7}(\mathbb Q_{2})$ introduced
in Section  \ref{self-dual}. Then the Zariski closure of $r_{\Pi}$ is 
$\SO_{7}(\bar{\mathbb Q}_{\ell})$. 
\end{Cor}
\begin{proof} This is easy since the image of the inertia subgroup $I_{2}$ contains 
elements in $\SO_{7}(\bar{\mathbb Q}_{\ell})$ which are not contained in $\G_{2}$.
\end{proof}

 \section{A group-theoretic criterion} \label{exceptional}

In this section we develop certain criteria which give us control over the image of 
$\ell$-adic representations in the case of exceptional groups. As such, this section 
 is somewhat more general then what is needed for the main results in this paper.
 However,   the results of this section
  might have future applications. A possibility in this direction is 
 presented in Section \ref{last}. 
 
  \smallskip 
 
Let $\Gamma$ be a profinite 
group and $d\ge 2$ an integer.   We define $\Gamma^d$ as the intersection of
all open normal subgroups of $\Gamma$ of index $\le d$.  

\begin{lemma}
\label{d-commutes}
If $\Gamma$ is a profinite group, $\Delta$ a closed normal subgroup, and $d$ a positive integer,
then the image of $\Gamma^d$ in $\Gamma/\Delta$ is $(\Gamma/\Delta)^d$.
\end{lemma}

\begin{proof}
The image in $\Gamma/\Delta$ of every open subgroup of $\Gamma$ of index $\le d$
is again open of index $\le d$, and conversely, all open index $\le d$ subgroups of $\Gamma/\Delta$
arise as images of open index $\le d$ subgroups of $\Gamma$.
\end{proof}

Let $n\ge 2$ be an integer and $p$ a prime congruent to $1$ (mod $n$). 

\begin{definition}\label{def}
 By a
group of \emph{type $(n,p)$}, we mean any finite group $\Gamma$ with a normal subgroup $\Delta$ isomorphic
to $\Z/p\Z$ such that the image of $\Inn \Gamma$ in $\Aut \Delta$ is isomorphic to $\Z/n\Z$.
\end{definition}

As noted in the introduction, this is slightly different from the terminology in \cite{KLS}.  If $\ell\neq p$ is prime,
a group of \emph{type $(n,p,\ell)$} will mean a (possibly finite) profinite group which is the 
extension of a group of type $(n,p)$ by a pro-$\ell$ group.

\begin{lemma}
\label{quotient}
If $0\to \Gamma_1\to \Gamma_2\to\Gamma_3\to 0$ is a short exact sequence of profinite groups, 
$n\ge 2$, $\ell$ and $p$ are distinct primes, and $\Gamma_1$ is pro-$\ell$, then $\Gamma_2$ contains a subgroup of type $(n,p,\ell)$
if and only if $\Gamma_3$ does.
\end{lemma}

\begin{proof}
Any extension of a group of type $(n,p,\ell)$ by a pro-$\ell$ group is again of type $(n,p,\ell)$,
so one direction is trivial.

For the other, let $\Delta_2$ be a closed subgroup of $\Gamma_2$ and $\Delta'_2$ an open
pro-$\ell$ subgroup of $\Delta_2$ such that $\Delta''_2:=\Delta_2/\Delta'_2$ is of type $(n,p)$.
Let $\Delta_1 := \Delta_2\cap\Gamma_1$, $\Delta'_1 := \Delta'_2\cap\Gamma_1$, and 
$\Delta''_1 := \Delta_1/\Delta'_1$.
By the snake lemma, we have a right-exact sequence
$$\Delta'_2/\Delta'_1\to \Delta_2/\Delta_1\to\coker(\Delta''_1\to \Delta''_2)\to 0.$$
As $\Delta'_2$ is pro-$\ell$, so is every quotient thereof, so $\Delta_2/\Delta_1$
is the extension of $\coker(\Delta''_1\to \Delta''_2)$ by a pro-$\ell$ group.

Every quotient of a group $\Gamma_{n,p}$ of type $(n,p)$ by an $\ell$-group
is again of type $(n,p)$.
Indeed, the quotient map preserves the normal subgroup of $\Gamma_{n,p}$ isomorphic to 
$\Z/p\Z$ and therefore the image of $\Inn\Gamma_{n,p}\to(\Z/p\Z)^\times$.
\end{proof}

We remark that for the non-trivial direction, the proof uses only the fact that $\Gamma_1$ is the inverse limit of finite groups of prime-to-$p$ order.

\begin{Thm}
\label{subgroup}
Let $\ell$ be a prime and $G$ a connected reductive algebraic group over $\bar\F_\ell$.
There exists an absolute constant $B$ such that
\begin{enumerate}
\item If $\rk G \le 2$ and $p>B$ is a prime distinct from $\ell$ and $G(\bar\F_\ell)$ contains a
subgroup of type $(6,p,\ell)$, then $G$ is of type $\G_2$.
\item If $\rk G \le 4$ and $p_1,p_2>B$ are primes distinct from $\ell$ and $G(\bar\F_\ell)$ contains 
subgroups of type $(8,p_1,\ell)$ and $(12,p_2,\ell)$, then $G$ is of type $\rF_4$.
\item If $\rk G \le 6$ and $p>B$ is a prime distinct from $\ell$ and $G(\bar\F_\ell)$ contains a
subgroup of type $(9,p,\ell)$, then $G$ is of type $\E_6$.
\item If $\rk G \le 7$ and $p_1,p_2>B$ are primes distinct from $\ell$ and $G(\bar\F_\ell)$ contains 
subgroups of type $(18,p_1,\ell)$ and $(30,p_2,\ell)$, then $G$ is of type $\E_7$.
\item If $\rk G \le 8$ and $p_1,p_2,p_3>B$ are primes distinct from $\ell$ and $G(\bar\F_\ell)$ contains 
subgroups of type $(18,p_1,\ell)$, $(20,p_2,\ell)$, and  $(30,p_3,\ell)$, then $G$ is of type $\E_8$.
\end{enumerate}
\end{Thm}

\begin{proof}
We claim that the root systems of type $\G_2$, $\rF_4$, $\E_6$, $\E_7$, and $\E_8$ respectively
are the only root systems of rank less than or equal to 
$2$, $4$, $6$, $7$, and $8$ respectively, which have Weyl
group elements of order $6$; $8$ and $12$; $9$; $18$ and $30$; 
and $18$, $20$, and $30$ respectively.
To see that this is so, we compile a table of the orders of Weyl group elements for root systems
of rank $\le 8$.  
We write each root system as a sum of irreducible root systems, each coded as a single letter and a single digit, and arranged alphabetically.  We order root systems first by rank and within rank, alphabetically.  
A root system is \emph{exceptional} if it is simple of type $\G_2$, $\rF_4$, $\E_6$, $\E_7$, or $\E_8$.
For brevity, we omit those root systems for which the set of possible Weyl element orders is a proper subset of that for some non-exceptional root system of equal or inferior rank.  In case of equality, we print only the
lexicographically smallest example.
(For example, in our table, no root system of type $C_n$ appears, since 
$B_n$ is lexicographically inferior to it and has the same set of Weyl group orders.
Likewise, $A_1+A_4$ does not appear because its set of Weyl group orders is strictly dominated by that of $B_5$, which, though lexicographically superior, has the same rank.)
Since the set of orders of elements of a finite group is determined by its subset of maximal elements with respect to divisibility, we exhibit only this subset.

\tablehead{\hline}
\tabletail{\hline}
\begin{center}
\begin{supertabular}{|c|l|}
\hline Root System&Maximal Elements\\
\hline A1&2 \\ 
\hline A2&2,3 \\ B2&4 \\ G2 & 6 \\ 
\hline B3&4,6 \\ 
\hline A2+B2&12 \\ A4&4,5,6  \\ B4&4,6,8 \\ F4&8,12 \\
\hline B5&8,10,12 \\ 
\hline A2+B4&24 \\ A4+B2&12,20 \\ A4+G2&12,30 \\ A6&7,10,12 \\ E6&8,9,10,12 \\ 
\hline A1+E6&8,10,12,18 \\ A2+B5&24,30 \\ A4+B3&12,20,30 \\
A7&7,8,10,12,15 \\ B7&14,20,24 \\ E7&8,12,14,18,30 \\
\hline A2+E6&18,24,30 \\A4+F4&24,40,60 \\A6+B2&12,20,28  \\A6+G2&12,30,42\\
A8&8,9,12,14,15,20 \\ B2+E6&8,20,36\\B8&14,16,20,24,30 \\ E8&14,18,20,24,30 \\
\hline
\end{supertabular}
\end{center}

Now, let $\Gamma_{n,p,\ell}$ be a subgroup of $G(\bar\F_\ell)$ of type $(n,p,\ell)$ for some $n\ge 2$
and some prime $p\neq \ell$.  Let $\Delta_\ell\subset \Gamma_{n,p,\ell}$ denote a normal
$\ell$-subgroup such that the corresponding quotient group $\Gamma_{n,p}$ is of type
$(n,p)$.   Now, $\Delta_\ell\subset G(\F_{\ell^k})$ for some $k$, and so it is contained in
a Sylow $\ell$-subgroup of $G(\F_{\ell^k})$.  Such a subgroup is the group of
$\F_{\ell^k}$-points of the unipotent radical of a Borel subgroup of $G$.  By \cite[30.3]{Hu1},
the normalizer of $\Delta_\ell$ in $G$ is a parabolic subgroup $P$, proper if $\Delta_\ell$ is non-trivial.  If $N$ denotes the unipotent
radical of $P$, then $N(\bar\F_\ell)\cap \Gamma_{n,p,\ell}$ is an $\ell$-group, 
so by Lemma~\ref{quotient},
the image of $\Gamma_{n,p,\ell}$ in the Levi factor $M(\bar\F_\ell)$ is again of type $(n,p,\ell)$.
The rank of $M$ is equal to that of $G$, while the dimension is strictly less.  Iterating this process
we end up with a connected reductive group, which we again denote $M$, 
of the same rank as $G$ such that
$M(\bar\F_\ell)$ contains a subgroup $\Gamma_{n,p}$ of type $(n,p)$.   If $M$ is an exceptional group, then $G=M$, since in each rank $r$, the exceptional group, if one exists, is the connected reductive group of maximal dimension in rank $r$.

Let $\Gamma_{n,p}$ be a subgroup of type $(n,p)$ of $M(\bar\F_\ell)$ for some integer $n$, let
$x$ denote the image of a generator of the normal subgroup of  $\Gamma_{n,p}$ isomorphic to $\Z/p\Z$,
and let $a\in \Z$ be such that its image in $\Z/p\Z$ is of order $n$ in $(\Z/p\Z)^\times$.  
As $x$ and $x^a$ are conjugate in $\Gamma_{n,p}$, $x$ and $x^a$ are conjugate in
$M(\bar\F_\ell)$.  Since $p\neq \ell$, they are semisimple elements, and if $B<p$ is taken to be larger than the maximal number of components of the centralizer of any semisimple element in any reductive connected group of rank $\le 8$, it follows that $x$ and $x^a$ belong to a common maximal torus
$T\subset M$.  By a well-known theorem \cite[\S3.1]{Hu2}, there exists $w\in N_M(T)(\bar\F_\ell)$
such that $wxw^{-1} = x^a$.  However, this implies that the order of the image of $w$ in the Weyl group
of $M$ with respect to $T$ is divisible by $n$.   From our analysis of orders of elements in Weyl groups in rank $\le 8$, it follows that $M$ is exceptional and therefore that $G$ is exceptional.

\end{proof}

\begin{Thm}\label{criterion}
There exist constants $A$ and $B$ such that if 
\begin{enumerate}
\item $d > A$ is an integer, 
\item $p_1,p_2,p_3 > B$ and $\ell\not\in\{p_1,p_2,p_3\}$ are primes,
\item $K$ is an $\ell$-adic field, 
\item $G$ is a connected reductive algebraic group over $K$ such that
\begin{enumerate}
\item $\rk G\le 2$,
\item $\rk G\le 4$,
\item $\rk G\le 6$,
\item $\rk G\le 7$, or
\item $\rk G\le 8$,
\end{enumerate}
\item $\Gamma\subset G(K)$ is a profinite subgroup such that (respectively)
\begin{enumerate}
\item $\Gamma^d$ has a subgroup of type $(6,p_1,\ell)$;
\item $\Gamma^d$ has a subgroup of type $(8,p_1,\ell)$ and $(12,p_2,\ell)$;
\item $\Gamma^d$ has a subgroup of type $(9,p_1,\ell)$;
\item $\Gamma^d$ has subgroups of type $(18,p_1,\ell)$ and $(30,p_2,\ell)$; or
\item $\Gamma^d$ has subgroups of type $(18,p_1,\ell)$, $(20,p_2,\ell)$, and $(30,p_3,\ell)$,
\end{enumerate}
\end{enumerate}
then some finite quotient $\bar\Gamma$ of $\Gamma$ satisfies
\begin{equation}
\label{squeeze}
(H^\ad(\bar\F_\ell)^F)^\der\subset \bar\Gamma\subset H^\ad(\bar\F_\ell)^F,
\end{equation}
where $F$ is a Frobenius map and $H^\ad$ is a simple adjoint algebraic group of type
$\G_2$, $\rF_4$, $\E_6$, $\E_7$, or $\E_8$ respectively.
In particular, in the first, second, and fifth case, $\bar \Gamma$ is simple.
\end{Thm}

\begin{proof}
Replacing $K$ by a finite extension, we may assume first that $G$ is split, and second that
$\Gamma$ fixes a hyperspecial vertex of the building of $G$ over $K$ (see, e.g., \cite{La}, and the remark below).  
Thus, there exists a smooth group scheme $\cG$ over the ring of integers $\cO$ of $K$ whose generic fiber is $G$, whose special fiber is again reductive and connected, and such that
$\Gamma\subset \cG(\cO)$.  The root datum of the special fiber of $\cG$ is the same as that of $G$.
Let $H := \cG_{\bar\F_\ell}$ denote the geometric special fiber.  Thus, $\Gamma$ maps to
$H(\bar\F_\ell)$ with finite image and pro-$\ell$ kernel.   Replacing $\Gamma$ by its image in
$H(\bar\F_\ell)$, by Lemma~\ref{d-commutes} and Lemma~\ref{quotient}, we still have that
$\Gamma^d$ has subgroups of the specified types.  By Theorem~\ref{subgroup},
$H$ is almost simple of type $\G_2$, $\rF_4$, $\E_6$, $\E_7$, or $\E_8$ according as we are in case
(a), (b), (c), (d), or (e).  Assuming $p>3$, the image $\bar\Gamma$ of $\Gamma$ in 
$H^\ad(\bar\F_\ell)$ again has the property that $\bar\Gamma^d$ has subgroups of the specified types.  By \cite[Th.~0.5]{LP}, it follows that either $\bar\Gamma$ satisfies the condition (\ref{squeeze})
for some Frobenius map or that $\bar\Gamma$ is contained in a proper algebraic subgroup 
of $I\subset H^\ad$ with a component group $I/I^\circ$ 
whose order is bounded above by an absolute constant $A$.
As $d>A$, it follows that $\bar\Gamma^d\subset I^\circ(\bar\F_\ell)$.  If $N$ denotes the unipotent radical of $I^\circ$, the image of $\bar\Gamma^d$ in $(I/N)(\bar\F_\ell)$ contains a subgroup of type
$(n,p,\ell)$.  As $I/N$ is reductive of rank less than or equal to the rank $r$ of $H^\ad$ (which equals the rank of $G$) and as $\dim I/N\le \dim I < \dim H^\ad = \dim G$, it follows that $I/N$ cannot be exceptional of rank $r$, which contradicts Theorem~\ref{subgroup}.

\end{proof}

\smallskip 
\noindent 
{\bf Remark:} If $G$ is of type $\G_2$, $\rF_4$ or $\E_8$ then $G(K)$ is always split and 
simply connected. The profinite subgroup $\Gamma$ is contained in a maximal parahoric
subgroup of $G(K)$. The quotient of this maximal parahoric subgroup by its  
pro-$\ell$ radical is a simply connected semisimple group $H$ of rank $r$ over the 
residual field of $K$. Let $\bar{\Gamma}$ be the projection of $\Gamma$ into $H$. If
$\Gamma$ contains groups of type  $(n, p, \ell)$ as specified in Theorem \ref{criterion}, 
then so does $\bar{\Gamma}$. By Theorem \ref{subgroup} $H$ must be of the same type as $G$. 
This shows that the maximal parahoric subgroup containing $\Gamma$ is hyperspecial.

\section{Main Theorem}

We are now ready to construct finite Galois groups over $\mathbb Q$. 
Let $r_{\Pi}: G_{\mathbb Q} \rightarrow G(\bar{\mathbb Q}_{\ell})$ be the Galois 
representation attached to 
a self-dual cuspidal representation $\Pi$ constructed in Section \ref{Global}.
Recall that $\Pi$ is constructed so that the image of $D_q$ in
$\Gamma=r_{\Pi}(G_{\mathbb Q})$ is a group of type $(2n,p)$, denoted by $\Gamma_{2n,p}$, and  
 contained in $\Gamma^d$. By Theorem \ref{zariski} and its corollaries, if $d>I(n)$ then 
$G=\G_{2}$ if $\Pi$ is a lift from $\G_{2}$ and  $G=\SO_{2n+1}$ if $\Pi$ is a lift 
from $\Sp_{2n}$ and the local component $\Pi_{2}$ is the supercuspidal 
representation defined in Section \ref{depth1}.
 In particular, $\ell>2$ if $G=\SO_{2n+1}$. If $G=\G_{2}$ then by 
Theorem \ref{criterion} $\Gamma$  has a quotient isomorphic to 
$\G_{2}(\mathbb F_{\ell^{k}})$ (or a Ree group if $\ell=3$) for some $k$ provided that $d>A,B$ where 
$A$ and $B$ are in the statement of Theorem \ref{criterion}.  
 
 Assume now that $G=\SO_{2n+1}$ and that $\Pi$ is a lift 
from $\Sp_{2n}$ and the local component $\Pi_{2}$ is the supercuspidal 
representation defined in Section \ref{depth1}.

\begin{Thm} \label{classical} Assume that $\ell > 2$. 
There exists a function $d : \mathbb N \rightarrow \mathbb N$ such that 
if $\Pi_q$ is picked with $d>d(n)$ then $\Gamma=r_{\Pi}(G_{\mathbb Q})$ has a quotient $\bar\Gamma$ such that  
\begin{equation}
\label{squeeze1}
(\SO_{2n+1}(\bar{\mathbb F}_\ell)^F)^\der\subset \bar\Gamma\subset 
\SO_{2n+1}(\bar{\mathbb F}_\ell)^F,
\end{equation}
where $F$ is a Frobenius map.  
\end{Thm}  
\begin{proof} 
 By enlarging the field $K$ we can assume that 
$G$ is split and that $\Gamma$ is contained in a hyperspecial maximal 
parahoric subgroup.
This means that $G$ can be written as $G=\SO(V,Q)$ where $V$ is a linear space 
over $K$ and $Q$ a split quadratic form, and there exists a lattice $L$ stabilized by $\Gamma$ such 
that $Q$ takes integral values on $L$. Moreover, if $\bar V$ is the reduction modulo $\ell$ of 
$L$, then the quadratic form $Q$ reduces modulo $\ell$ to a non-degenerate quadratic form  
$\bar Q$ on $\bar V$.   In other words, the pair $(L,Q)$ defines a  
smooth group scheme  over the ring of integers  of $K$ whose generic fiber is $G$ and such that 
$H := \SO(\bar V, \bar Q)$ is the special fiber. 
Let $\bar\Gamma$ be the image of $\Gamma$ in
$H(\bar{\mathbb F}_\ell)$.  
By \cite[Th.~0.5]{LP}, it follows that either $\bar\Gamma$ satisfies the condition (\ref{squeeze1})
for some Frobenius map or that $\bar\Gamma$ is contained in a proper algebraic subgroup
of $I\subset H$ with a component group $I/I^\circ$
whose order is bounded above by an absolute constant $A(n)$. If we pick $d> A(n)$ then 
$\bar\Gamma^d\subset I^\circ(\bar{\mathbb F}_\ell)$. It follows that $\Gamma_{2n,p}$ is contained in $I^\circ$. 
Under the action of $\Gamma_{2n,p}$, the orthogonal space $\bar V$ decomposes as a sum of two 
irreducible mutually orthogonal representations of dimensions $2n$ and $1$. This implies that the 
nilpotent radical of $I^\circ$ is trivial. Indeed, if $N$ is a non-trivial unipotent radical of $I^\circ$,then 
there exists a non-trivial subspace $\bar U$ of $\bar V$ fixed by $N$.  
Since $\Gamma_{2n,p}$ normailizes $N$, $\bar U$ must be one of the two mutually orthogonal 
$\Gamma_{2n,p}$-summands. By orthogonality, $N$ must preserve the other summand. Since that summand 
is also $\Gamma_{2n,p}$-irreducible, $N$ must be trivial, a contradiction.   
Thus, $I^\circ$ is reductive. We claim that $\bar V$ is an irreducible $I^\circ$-module. If not, then 
$I^\circ$ admits a $2n$-dimensional orthogonal representation such that there exists 
a Weyl group element in $I^\circ$ permuting transitively all weights. We need the following: 

\begin{lemma} Let $G$ be a connected reductive group over an algebraically closed field of
 characteristic $\neq 2$. 
Then $G$ has no irreducible orthogonal representations of even dimension such that there exists a Weyl group 
element permuting all weights. 
\end{lemma}
\begin{proof} 
Of course, we can assume that $G$ is semi-simple and that $G=G_1 \times \cdots \times G_k$, a product of 
almost simple groups. If $V$ is a miniscule, self-dual representation of $G$ then $V=V_1\otimes \cdots \otimes V_k$ 
where $V_i$ are miniscule and self-dual representations of $G_i$. Of course, if $G$ contains a Weyl group element 
permuting all weights in $V$, then the same holds for all pairs $(G_i,V_i)$. The converse is true only if 
the dimensions of $V_i$ are pairwise relatively prime. If $G$ is simple, an argument in \cite{KLS} shows that the 
only minuscule, self-dual representations with such Weyl group element is the standard 
representation of $\Sp_{2n}$ and its Frobenius twists. (The list there includes also the standard 
representation of $\SO_{2n}$, but that one can be excluded by a direct inspection.)
 Since even numbers are never pairwise relatively prime, 
we can conlude that $G$ is simple and that the representation is the standard representation of $\Sp_{2n}$. 
This one is not orthogonal, however, and the lemma follows. 
\end{proof} 

The lemma implies that $\bar V$ is an irreducible representation of $I^\circ$. 
Therefore, the pair $(I^\circ , \bar V)$ satisfies 
 satisfies conditions of Lemma \ref{almost}. It follows that  
$I^{\circ}=H$ or, if $n=3$, $I^{\circ}=\G_{2}(\bar{\mathbb F}_{\ell})$. 
Since $\G_{2}$ has the trivial center, no outer automorphisms and it acts irreducibly 
on $\bar V$, any element of $I/I^{\circ}$ is represented by a scalar matrix. However, 
by \v Cebotarev's density theorem, any element in $\bar\Gamma$ has $1$ as an eigenvalue. 
It follows that $I=I^{\circ}$. Since the image of the local decomposition group $D_{2}$
contains elements of order 2 which are not contained in $\G_{2}$ we see that $I$ cannot be 
$\G_{2}$. Thus we have $I^{\circ}=H$ in all cases. This is a contradiction.  The theorem 
is proved with $d(n)=\max(A(n), I(n))$. 
\end{proof} 

Summarizing, we have shown that mod $\ell$ reduction of the representations 
$r_{\Pi}$ give rise to $\G_{2}(\mathbb F_{\ell^{k}})$ (or a Ree group if $\ell=3$), 
$\SO_{2n+1}(\mathbb F_{\ell^{k}})^{\der}$ or 
$\SO_{2n+1}(\mathbb F_{\ell^{k}})$ as Galois group.
In other words we have essentially proved the following theorem that we stated 
in the introduction: 

\begin{Thm} Let $t$ be a positive integer. We take $t$ to be even if $\ell=3$ in the 
first case below. 
\begin{enumerate}
\item  Let $\ell$ be a prime. Then there exists an integer $k$ divisible by $t$ such that 
the simple group $\G_{2}(\mathbb F_{\ell^{k}})$ appears as a Galois group over
$\mathbb Q$. 
\item Let $\ell$ be an odd prime. Then there exists an integer $k$ divisible by $t$ such that 
 the finite simple group $\SO_{2n+1}(\mathbb F_{\ell^{k}})^{\der}$ or the 
finite classical group $\SO_{2n+1}(\mathbb F_{\ell^{k}})$
appears as a Galois group over $\mathbb Q$. 
\item If $\ell \equiv 3, 5\pmod{8}$,  then there exists an integer $k$ divisible by $t$ such that  the finite simple group 
$\SO_{2n+1}(\mathbb F_{\ell^{k}})^{\der}$ 
appears as a Galois group over $\mathbb Q$. 
\end{enumerate}
\end{Thm}
\begin{proof} The divisibility of $k$ by $t$ follows from part (3) of Lemma \ref{deep}.
The remark following Lemma \ref{deep} shows that by taking $t$ even if $\ell=3$
in the first case eliminates Ree groups. It remains to deal with the third 
statement. Assume now that the Galois group given by part (2)
 is $\SO_{2n+1}(\mathbb F_{\ell^{k}})$.  Then the subgroup 
$\SO_{2n+1}(\mathbb F_{\ell^{k}})^{\der}$ of index 2 defines a quadratic field 
$K$. This field is unramified for all primes different from $2,\ell$ and $q$, since the 
same holds for the representation $r_{\Pi}$. Since the image of the inertia $I_{q}$ is of order $p$, 
it lands in the subgroup $\SO_{2n+1}(\mathbb F_{\ell^{k}})^{\der}$.  Thus, 
$K$ is unramified at $q$ also. Moreover, by 
Proposition \ref{two}, the image of the decomposition group $D_{2}$ is a group such that 
the quotient by its commutator is odd. Such group must be contained in 
the subgroup $\SO_{2n+1}(\mathbb F_{\ell^{k}})^{\der}$. This shows not only that 
that $K$ is unramified at $2$ but $2$ splits in $K$. We remind the reader that the unique 
quadratic field ramified at $\ell$ and no other primes is 
 $\mathbb Q(\sqrt{\ell})$ if $\ell\equiv 1\pmod{4}$ 
and $\mathbb Q(\sqrt{-\ell})$ if $\ell\equiv 3\pmod{4}$. However, 
since $2$ splits in this field if and only if $\ell \equiv 1,7 \pmod{8}$ 
we see that if $\ell \equiv 3,5 \pmod{8}$ the 
Galois group constructed in part (2) is in fact  $\SO_{2n+1}(\mathbb F_{\ell}^{k})^{\der}$.  
 \end{proof}

\section{On future directions}\label{last}

One difficulty that we needed to address in this paper came from the fact that
the group  $\GL_{2n+1}(\mathbb Q_{p})$
has no self-dual supercuspidal representation unless $p=2$. 
In order to construct Galois groups of type $\B_{n}$ this problem was resolved by 
introducing the self-dual supercuspidal representation $\Pi_{2}$ of 
$\GL_{2n+1}(\mathbb Q_{2})$ whose parameter contains a Jordan subgroup of 
$\SO_{2n+1}(\mathbb C)$. 
For Galois groups of type $\G_{2}$ the construction is based on a technical improvement 
of Theorem \ref{reciprocity} due to Shin, which is based on the fundamental lemma 
for unitary group. Another way, which avoids the use of the fundamental 
lemma, would be to pick $\Pi_{2}$ so that its parameter comes from 
the Jordan subgroup in $\G_{2}$. More precisely the parameter of $\Pi_{2}$ 
should be the homomorphism $\phi_{2}: W_{\mathbb Q_{2}}\rightarrow 
\G_{2}$ described in Proposition \ref{jordan}. In order to obtain a global lift 
from $\G_{2}$ to $\GL_{7}$ with this $\Pi_{2}$ as a local component one would 
need to complete the following (doable) program: 
\begin{itemize}
\item Define, via induction from open compact subgroup,
 a generic supercuspidal $\sigma_{2}$ representation of $\G_{2}(\mathbb Q_{2})$ 
corresponding to the parameter $\phi_{2}$. 
\item Compute the theta lift of $\sigma_{2}$ to $\PGSp_{6}(\mathbb Q_{2})$. 
\item Show, using the method of \cite{Sa2}, that the further lift to $\GL_{7}(\mathbb Q_{2})$
is $\Pi_{2}$. 
\end{itemize}

\smallskip 

A construction of supercuspidal representations attached to parameters arising 
from Jordan subgroups is a subject of the forthcoming paper by Gross and Reeder \cite{GR}.
 The completion of the three step program would give Galois groups of type $\G_{2}$ except in 
 the residual characteristic 2 without using results of Shin \cite{Sh}. 
 There is yet another approach which removes the restriction $\ell \neq 2$, but
   introduces a different conjecture. 
Let $G$ denote $\Sp_{2n}$ or $\G_{2}$. Then using the trace formula it is possible to show that 
there exist a cuspidal automorphic representation $\sigma$ of $G(\mathbb A)$ 
unramified at all primes different from $\ell,q$ and such that 
\begin{itemize}
\item $\sigma_{\infty}$ is a discrete series representation with a large (unspecified) 
parameter (weight). 
\item $\sigma_{q}$ is a specified supercuspidal representation. 
\item $\sigma_{\ell}$ is the Steinberg representation. 
\end{itemize}
\emph{Assuming} that $\sigma$ is globally generic
then $\Pi$, the lift of $\sigma$ to $\GL_{2n+1}(\mathbb A)$, is 
automatically cuspidal and has a discrete series representation at one 
local place, since $\Pi_{\ell}$ is the Steinberg representation of 
$\GL_{2n+1}(\mathbb Q_{\ell})$. (We use here that the theta lift of the Steinberg 
representation of $\G_{2}(\mathbb Q_{\ell})$ is the Steinberg representation of 
$\PGSp_{6}(\mathbb Q_{\ell})$, see \cite{GrS}.) 
Since matrix coefficients of the Steinberg representation 
are in $L^{1+\epsilon}(G)$ and therefore not integrable, we note that 
the method of Poincar\'e series cannot be used to  construct such $\sigma$.

\smallskip 

In principle our method could be extended to other groups. The main limitation at 
the moment is the lack of $\ell$-adic representations attached to automorphic 
representations. If we assume, for example, that one can attach a $26$-dimensional 
$\ell$-adic representation to an algebraic automorphic form of the exceptional 
group ${\rm F}_{4}$ then we would be able to construct finite groups of type ${\rm F}_{4}$ as 
Galois groups over $\mathbb Q$. Indeed, to this end one would pick a  cuspidal 
automorphic form $\sigma$ of ${\rm F}_{4}$ such that for two primes $p_{1}$ and $p_{2}$ the local 
components $\sigma_{p_{1}}$ and $\sigma_{p_{2}}$ are tame supercuspidal representations 
whose parameters have groups of type $(8, p_{1})$ and $(12, p_{2})$, respectively, as the 
image. Thus,  if the two parameters at $p_{1}$ and $p_{2}$ are 
picked so that the images of the local decomposition 
groups  are deeply embedded, then the results of Section \ref{exceptional}
imply that the restriction modulo $\ell$ of the 
$\ell$-adic representation attached to $\sigma$ will give finite groups of type 
${\rm F}_{4}(\ell^{k})$ as Galois groups over $\mathbb Q$.

\end{document}